\documentclass[11pt]{amsart}
\usepackage{amscd}
\usepackage[all]{xy}
\usepackage{graphicx}
\usepackage{amsmath}
\usepackage{amsfonts}
\usepackage{amssymb}
\usepackage{latexsym}
\usepackage{slashed}
\usepackage{soul}
\usepackage{comment}
\usepackage{color}

\usepackage{amsmath, latexsym, amssymb}
\numberwithin{equation}{section}
\theoremstyle{plain}
\newtheorem{lemma}{Lemma}[section]
\newtheorem{proposition}[lemma]{Proposition}
\newtheorem{theorem}[lemma]{Theorem}

\theoremstyle{definition}
\newtheorem{definition}[lemma]{Definition}
\newtheorem{remark}[lemma]{Remark}
\newtheorem{example}[lemma]{Example}

\begin{document}
\newcommand{\R}{{\mathbb R}}
\newcommand{\C}{{\mathbb C}}
\newcommand{\F}{{\mathbb F}}
\renewcommand{\O}{{\mathbb O}}
\newcommand{\Z}{{\mathbb Z}} 
\newcommand{\N}{{\mathbb N}}
\newcommand{\Q}{{\mathbb Q}}
\newcommand{\Sb}{{\mathbb S}}
\renewcommand{\H}{{\mathbb H}}
\renewcommand{\P}{{\mathbb P}}
\newcommand{\Aa}{{\mathcal A}}
\newcommand{\Bb}{{\mathcal B}}
\newcommand{\Cc}{{\mathcal C}}    %configuration space
\newcommand{\Dd}{{\mathcal D}}
\newcommand{\Ee}{{\mathcal E}}
\newcommand{\Ff}{{\mathcal F}}
\newcommand{\Gg}{{\mathcal G}}    %gauge transformations
\newcommand{\Hh}{{\mathcal H}}
\newcommand{\Kk}{{\mathcal K}}
\newcommand{\Ii}{{\mathcal I}}
\newcommand{\Jj}{{\mathcal J}}
\newcommand{\Ll}{{\mathcal L}}    %Loop space
\newcommand{\Mm}{{\mathcal M}}    %moduli space
\newcommand{\Nn}{{\mathcal N}}
\newcommand{\Oo}{{\mathcal O}}
\newcommand{\Pp}{{\mathcal P}}
\newcommand{\Qq}{{\mathcal Q}}
\newcommand{\Rr}{{\mathcal R}}
\newcommand{\Ss}{{\mathcal S}}
\newcommand{\Tt}{{\mathcal T}}
\newcommand{\Uu}{{\mathcal U}}
\newcommand{\Vv}{{\mathcal V}}
\newcommand{\Ww}{{\mathcal W}}
\newcommand{\Xx}{{\mathcal X}}
\newcommand{\Yy}{{\mathcal Y}}
\newcommand{\Zz}{{\mathcal Z}}

\newcommand{\Ds}{{\slashed D}}

\newcommand{\zt}{{\tilde z}}
\newcommand{\xt}{{\tilde x}}
\newcommand{\Ht}{\widetilde{H}}
\newcommand{\ut}{{\tilde u}}
\newcommand{\Mt}{{\widetilde M}}
\newcommand{\Llt}{{\widetilde{\mathcal L}}}
\newcommand{\yt}{{\tilde y}}
\newcommand{\vt}{{\tilde v}}
\newcommand{\Ppt}{{\widetilde{\mathcal P}}}
\newcommand{\bp }{{\bar \partial}} 
\newcommand{\ad}{{\rm ad}}
\newcommand{\Om}{{\Omega}}
\newcommand{\om}{{\omega}}
\newcommand{\eps}{{\varepsilon}}
\newcommand{\Di}{{\rm Diff}}
\renewcommand{\Im}{{ \rm Im \,}}
\renewcommand{\Re}{{\rm Re \,}}

\renewcommand{\frak}[1]{{\mathfrak {#1}}}
\renewcommand{\a}{{\mathfrak a}}
\renewcommand{\b}{{\mathfrak b}}
\newcommand{\e}{{\mathfrak e}}
\renewcommand{\k}{{\mathfrak k}}
\newcommand{\m}{{\mathfrak m}}
\newcommand{\pg}{{\mathfrak p}}
\newcommand{\g}{{\mathfrak g}}
\newcommand{\gl}{{\mathfrak {gl}}}
\newcommand{\h}{{\mathfrak h}}
\renewcommand{\l}{{\mathfrak l}}
\newcommand{\sm}{{\mathfrak m}}
\newcommand{\n}{{\mathfrak n}}
\newcommand{\s}{{\mathfrak s}}
\renewcommand{\o}{{\mathfrak o}}
\renewcommand{\so}{{\mathfrak{so}}}
\newcommand{\spin}{{\mathfrak {spin}}}
\renewcommand{\u}{{\mathfrak u}}
\newcommand{\su}{{\mathfrak su}}
\newcommand{\X}{{\mathfrak X}}

\newcommand{\ssl}{{\mathfrak {sl}}}
\newcommand{\ssp}{{\mathfrak {sp}}}
\renewcommand{\t}{{\mathfrak t }}
\newcommand{\Cinf}{C^{\infty}}
\newcommand{\la}{\langle}
\newcommand{\ra}{\rangle}
\newcommand{\ha}{\scriptstyle\frac{1}{2}}
\newcommand{\p}{{\partial}}
\newcommand{\notsub}{\not\subset}
\newcommand{\iI}{{I}}               %unit interval [0,1]
\newcommand{\bI}{{\partial I}}      %boundary of same
\newcommand{\LRA}{\Longrightarrow}
\newcommand{\LLA}{\Longleftarrow}
\newcommand{\lra}{\longrightarrow}
\newcommand{\LLR}{\Longleftrightarrow}
\newcommand{\lla}{\longleftarrow}
\newcommand{\INTO}{\hookrightarrow}

\newcommand{\QED}{\hfill$\Box$\medskip}
\newcommand{\UuU}{\Upsilon _{\delta}(H_0) \times \Uu _{\delta} (J_0)}
\newcommand{\bm}{\boldmath}
\newcommand{\coker}{\mbox{coker}}

%Definitions Lorenz
\newcommand{\End}{\mbox{\rm End}}
\newcommand{\vol}{\mbox{\rm vol}}
\newcommand{\rmspan}{\mbox{\rm span}}
\newcommand{\codim}{\mbox{\rm codim}}
\newcommand{\Sp}{\mbox{\rm Spin}(7)}

\def\hook{\mbox{}\begin{picture}(10,10)\put(1,0){\line(1,0){7}}
 \put(8,0){\line(0,1){7}}\end{picture}\mbox{}}
\date{\today}
\title[Fr\"olicher-Nijenhuis bracket of $G_2$-and $\Sp$-manifolds]
{The Fr\"olicher-Nijenhuis bracket and the geometry of $G_2$-and $\Sp$-manifolds}

\author{Kotaro Kawai}
\address{Graduate School of Mathematical Sciences, University of Tokyo, 3-8-1, Komaba, Meguro, Tokyo 153-8914, Japan}
\email{kkawai@ms.u-tokyo.ac.jp}
\address{Current address: Gakushuin University, 1-5-1, Mejiro, Toshima,Tokyo, 171-8588, Japan}
\email{kkawai@math.gakushuin.ac.jp}

\author{H\^ong V\^an L\^e}
\address{Institute of Mathematics CAS
, Zitna 25, 11567 Praha 1, Czech Republic}
\email{hvle@math.cas.cz}

\author{Lorenz Schwachh\"ofer}
\address{Fakult\"at f\"ur Mathematik,
Technische Universit\"at Dortmund,
Vogelpothsweg 87, 44221 Dortmund, Germany}
\email{lschwach@math.tu-dortmund.de} 

\thanks{The first named author is supported Grant-in-Aid for JSPS fellows (26-7067),
 the second named author is partially supported by RVO: 67985840}

\begin{abstract} We extend the characterization of the integrability of an almost complex structure $J$ on differentiable manifolds via the vanishing of the Fr\"olicher-Nijenhuis bracket $[J, J]^{FN}$ to an analogous characterization of torsion-free $G_2$-structures and torsion-free $\Sp$-structures. We also explain the Fern\'andez-Gray classification of $G_2$-structures and the Fern\'andez classification of $\Sp$-structures in terms of the Fr\"olicher-Nijenhuis bracket.
\end{abstract}

\maketitle

%\tableofcontents

%%%%%%%%%%%%%%%%%%%%%%%%%%%%%%%%%%%%%%%%%%%%%%%%%%%%%%%%%
\section{Introduction}
%%%%%%%%%%%%%%%%%%%%%%%%%%%%%%%%%%%%%%%%%%%%%%%%%%%%%%%%%

A $G_2$-structure on a $7$-dimensional manifold $M^7$ is a $3$-form $\varphi \in \Om^3(M^7)$ which at each point $p \in M^7$ is contained in a certain open subset of $\Lambda^3 T^\ast_p M^7$; similarly, a $\Sp$-structure on an $8$-dimensional manifold $M^8$ is given by a $4$-form $\Phi \in \Om^4(M^8)$ which at each point is contained in a certain subset of $\Lambda^4 T^\ast_p M^8$. Such structures induce both an orientation and a Riemannian metric on the underlying manifold, denoted by $g_\varphi$ and $g_\Phi$, respectively, and $\Phi$ is self-dual w.r.t. this metric.

Manifolds with $G_2$-structures have first been investigated by Fern\'andez and Gray \cite{FG1982}, and $\Sp$-structures by Fern\'andez \cite{Fernandez1986} who showed that the covariant derivatives $\nabla \varphi \in \Om^1(M^7, \Lambda^3 T^\ast M^7)$ and $\nabla \Phi \in \Om^1(M^8, \Lambda^4 T^\ast M^8)$, respectively, decompose into four irreducible components in case of $G_2$-structures and into two irreducible components in the case of $\Sp$-structures. Thus, the conditions of the vanishing of some of these components yield $2^4 = 16$ classes of $G_2$-structure and $2^2 = 4$ classes of $\Sp$-structures, respectively, and the underlying geometries were discussed in \cite{FG1982} and \cite{Fernandez1986}; see also Section \ref{subs:16classesG2} below.

A $G_2$-structure ($\Sp$-structure, respectively) is called {\em torsion-free}, if $\varphi$ ($\Phi$, respectively) is parallel. As it turns out, the parallelity of $\varphi$ and $\Phi$, respectively, is equivalent to $\varphi$ and $\Phi$ being harmonic forms.

Alternatively, $G_2$- and $\Sp$-structures may be characterized via certain {\em ($2$-fold or $3$-fold) cross products on the tangent bundle}. These are given as the sections
\[
Cr_\varphi:= \delta_{g_\varphi} \varphi \in \Om^2(M^7, TM^7), \quad 
\chi_\varphi := -\delta_{g_\varphi} \ast \varphi \in \Om^3(M^7, TM^7)
\]
in case of $G_2$-structures, and as
\[
P_\Phi := -\delta_{g_\Phi} \Phi \in \Om^3(M^8, TM^8),
\]
where $\delta_g: \Om^{k+1}(M) \to \Om^k(M, TM)$ is the contraction of a differential form with the Riemannian metric $g$ and have natural interpretations via octonian multiplication. The triple cross product $\chi \in \Om^3(M^7, TM^7)$ on a manifold with a $G_2$-structure has been introduced by Harvey-Lawson \cite{HL1982} and was used in many papers on deformation of associative submanifolds, see e.g.\cite{McLean1998}, \cite{Kawai2014a}, \cite{LV2016}. The $3$-fold cross product $P$ on $\R^8$ has been first explicitly constructed by Brown and Gray \cite{BG1967}. They also proved that (up to the $G_2$-action) there are exactly two non-equivalent $3$-fold cross products on $\R^8=\O$. In \cite{HL1982} Harvey and Lawson intensively used the $3$-fold cross product on $\R^8$ which is related to the Cayley $4$-form and hence is invariant under the action of $\Sp$. Fernandez showed the uniqueness of a $\Sp$-invariant $4$-form on $\R^8$ (up to a multiplicative constant) and used the associated $3$-fold cross product to classify $\Sp$-structures on $8$-manifolds \cite{Fernandez1986}.

In this article, we view these cross products as elements of the {\em Fr\"olicher-Nijenhuis Lie algebra }$\Om^\ast(M, TM)$. Namely, it was shown by Fr\"olicher-Nijenhuis in \cite{FN1956} that $\Om^\ast(M, TM)$ can be given the structure of a graded Lie algebra using the {\em Fr\"olicher-Nijenhuis bracket $[\;,\; ]^{FN}$} in a natural way. Thus, given a manifold with a $G_2$-structure $(M^7, \varphi)$, we may consider the Fr\"olicher-Nijenhuis brackets
\[
{}[Cr_\varphi, \chi_\varphi]^{FN} \in \Om^5(M^7, TM^7), \quad [\chi_\varphi, \chi_\varphi]^{FN} \in \Om^6(M^7, TM^7),
\]
(observe that $[Cr, Cr]^{FN} = 0$ due to graded skew-symmetry), and analogously, for a manifold with a $\Sp$-structure $(M^8, \Phi)$ we may consider
\[{} [P_\Phi, P_\Phi]^{FN} \in \Om^6(M^8, TM^8).
\]
These brackets may be regarded as a natural generalization of the Nijenhuis tensor of an almost complex structure $J$. Indeed, regarding such a structure as an element $J \in \Om^1(M, TM)$, it turns out that $[J,J]^{FN} \in \Om^2(M, TM)$ coincides -- up to a constant multiple -- with the Nijenhuis tensor of $J$, whence $J$ is integrable if and only if $[J,J]^{FN} = 0$ \cite{FN1956b}.

Our main result is that the Fr\"olicher-Nijenhuis bracket also characterizes the torsion-freeness of $G_2$- and $\Sp$-structures, respectively. Namely, we show the following.

\begin{theorem} \label{thm:Brackets-intro}
Let $(M^7, \varphi)$ be a manifold with a $G_2$-structure and the associated Riemannian metric $g = g_\varphi$, and let $\nabla$ be the Levi-Civita connection of $g$. Then for every $p \in M^7$ the following are equivalent.
\begin{enumerate}
\item
The $G_2$-structure is torsion-free at $p$, i.e., $(\nabla \varphi)_p = 0$.
\item
$[Cr_\varphi, \chi_\varphi]_p^{FN} = 0 \in \Lambda^5 T_p^*M^7 \otimes T_pM^7$.
\item
$[\chi_\varphi, \chi_\varphi]_p^{FN} = 0 \in \Lambda^6 T_p^*M^7 \otimes T_pM^7$.
\end{enumerate}
\end{theorem}

In fact, we show in Theorem \ref{thm:Brackets} that $(\nabla \varphi)_p$ is characterized by either $[\chi_\varphi, \chi_\varphi]_p^{FN}$, or by the projection of $[Cr_\varphi, \chi_\varphi]_p^{FN}$ onto a subspace isomorphic to $T_pM^7 \otimes T_pM^7$.

\begin{theorem} \label{thm:Brackets-Spin-intro}
Let $(M^8, \Phi)$ be a manifold with a $\Sp$-structure, and let $\nabla$ be the Levi-Civita connection of the associated Riemannian metric $g = g_\Phi$. Then for every $p \in M^8$ the following are equivalent.
\begin{enumerate}
\item
The $\Sp$-structure is torsion-free at $p$, i.e., $(\nabla \Phi)_p = 0$.
\item
$[P_\Phi, P_\Phi]_p^{FN} = 0 \in \Lambda^6 T_p^*M^8 \otimes T_pM^8$.
\end{enumerate}
\end{theorem}

Namely, we show in Theorem \ref{thm:Brackets-Spin} that $(\nabla \Phi)_p$ is characterized by the projection of $[P_\Phi, P_\Phi]_p^{FN}$ onto a subspace isomorphic to $W^7_p
 \otimes T_pM^8$ for some rank-$7$ bundle $\Lambda^6 T^\ast M^8 \supset W^7 \to M^8$.

These explicit descriptions also allow us to give a complete characterization of the $16$ cases of $G_2$-structures in terms of $[Cr_\varphi, \chi_\varphi]_p^{FN}$ and $[\chi_\varphi, \chi_\varphi]_p^{FN}$, and of the $4$ classes of $\Sp$-structures in terms of $[P_\Phi, P_\Phi]_p^{FN}$; cf. Section \ref{subs:16classesG2}.

Our paper is organized as follows. In Section \ref{sec:pre} we recall the Fr\"olicher-Nijenhuis bracket on $\Om^\ast(M, TM)$. Then we turn to the case of $G_2$-structures in Section \ref{sec:g2}, characterizing the torsion endomorphism and showing the results that lead us to Theorem \ref{thm:Brackets-intro}. In Section \ref{sec:spin7}, we repeat this discussion for the case of $\Sp$-structures which leads to Theorem \ref{thm:Brackets-Spin-intro}. Finally, the characterization of the 16 classes of $G_2$-structures and the $4$ classes of $\Sp$-structures in terms of the Fr\"olicher-Nijenhuis bracket is given in Section \ref{subs:16classesG2}. The appendix then contains the proofs of some identities on representations of $G_2$ and $\Sp$ which are used throughout the paper.

%%%%%%%%%%%%%%%%%%%%%%%%%%%%%%%%%%%%%%%%%%%%%%%%%%%%%%%%%
\section{Preliminaries}\label{sec:pre}
%%%%%%%%%%%%%%%%%%%%%%%%%%%%%%%%%%%%%%%%%%%%%%%%%%%%%%%%%

%%%%%%%%%%%%%%%%%%%%%%%%%%%%%%%%%%%%%%%%%%%%%%%%%%%%%%%%%
\subsection{The Fr\"olicher-Nijenhuis bracket}\label{subs:fnb}
%%%%%%%%%%%%%%%%%%%%%%%%%%%%%%%%%%%%%%%%%%%%%%%%%%%%%%%%%

Let $M$ be a manifold and $(\Om^\ast(M), \wedge) = (\bigoplus_{k \geq 0} \Om^k(M), \wedge)$ be the graded algebra of differential forms. We shall use superscripts to indicate the degree of a form, i.e., $\alpha^k$ denotes an element of $\Om^k(M)$.

Evidently, contraction $\imath_X: \Om^k(M) \to \Om^{k-1}(M)$ with a vector field $X \in {\frak X}(M)$ is a derivation of degree $-1$. More generally, for $K \in \Om^k(M, TM)$ we define $\imath_K \alpha^l$ as the {\em contraction of $K$ with $\alpha^l \in \Om^l(M)$ }pointwise by
\[
\imath_{\kappa^k \otimes X} \alpha^l := \kappa^k \wedge (\imath_X \alpha^l) \in \Om^{k+l-1}(M),
\]
where $\kappa^k \in \Om^k(M)$ and $X \in {\frak X}(M)$ is a vector field, and this is a derivation of $\Om^*(M)$ of degree $l-1$. Thus, the {\em Nijenhuis-Lie derivative along $K \in \Om^k(M, TM)$ }defined as
\begin{equation} \label{eq:LK-deriv}
\Ll_K (\alpha^l) := [\imath_K, d] (\alpha^l) = \imath_K (d\alpha^l) + (-1)^k d(\imath_K \alpha^l) \in \Om^{k+l}(M)
\end{equation}
is a derivation of $\Om^*(M)$ of degree $k$.%, and evidently,
%\begin{equation} \label{eq:L-d-commute}
%\[
%\Ll_K(d\alpha) = (-1)^k d\Ll_K(\alpha).
%\]
%\end{equation}
%Moreover, if $K = \kappa_i^k \otimes e_i$ w.r.t. some local frame $(e_i)$, then
%\begin{equation} \label{eq:Lie-2}
%\Ll_K (\alpha^l) := \kappa_i^k \wedge \Ll_{e_i} \alpha^l + (-1)^k d\kappa_i^k \wedge \imath_{e_i} \alpha^l.
%\end{equation}

Observe that for $k = 0$ in which case $K \in \Om^0(M, TM)$ is a vector field, both $\imath_K$ and $\Ll_K$ coincide with the standard notion of contraction with and Lie derivative along a vector field.

In \cite{FN1956} \cite{FN1956b}, it was shown that $\Om^*(M, TM)$ can be given a unique Lie algebra structure, called the {\em Fr\"olicher-Nijenhuis bracket} and denoted by $[\cdot, \cdot]^{FN}$, such that $\Ll$ defines an action of $\Om^*(M, TM)$ on $\Om^*(M)$, that is,
\begin{equation} \label{eq:FN-homom}
\Ll_{[K_1, K_2]^{FN}} = [\Ll_{K_1}, \Ll_{K_2}] =: \Ll_{K_1} \circ \Ll_{K_2} - (-1)^{|K_1||K_2|} \Ll_{K_2} \circ \Ll_{K_1}.
\end{equation}
It is given by the following formula for $\alpha^k \in \Om^k(M)$, $\beta^l \in\Om ^l (M)$, $X_1, X_2 \in \X (M)$ \cite[Theorem 8.7 (6), p. 70]{KMS1993}:
\begin{align}
\nonumber [\alpha^k \otimes X_1, &\beta^l \otimes X_2]^{FN} = \alpha^k \wedge \beta^l \otimes [ X_1, X_2]\nonumber \\
& + \alpha^k \wedge \Ll_{X_1} \beta^l \otimes X_2 
- \Ll_{X_2} \alpha^k \wedge \beta^l \otimes X_1 \label{eq:kms}\\ 
&+ (-1)^{k} \left( d \alpha^k \wedge (\imath_{X_1} \beta^l) \otimes X_2 
+ (\imath_{X_2} \alpha^k) \wedge d \beta^l \otimes X_1 \right).
\nonumber
\end{align}
In particular, for a vector field $X \in \X(M)$ and $K \in \Om^*(M, TM)$ we have \cite[Theorem 8.16 (5), p. 75]{KMS1993}
\begin{align*}
\Ll_X (K) = [X, K] ^{FN},
\end{align*}
that is, the Fr\"olicher-Nijenhuis bracket with a vector field coincides with the Lie derivative of the tensor field $K \in \Om^*(M, TM)$ which means that $\exp(tX): \Om^*(M, TM) \to \Om^*(M, TM)$ is the action induced by (local) diffeomorphisms of $M$.

\begin{example} \label{ex:Lk-J}
Let $A \in \Om^1(M, TM)$ be an endomorphism field on $M$. Then \cite[Remark 8.17, p. 75]{KMS1993}
$$[A, A]^{FN} = 2 [A, A]_N,$$
where $[A, A]_N$ is the Nijenhuis tensor of $A$. W.r.t. a local frame $(e_i)$ with dual frame $(e^i)$ we can write $A = e^i \otimes A e_i$, whence
\begin{align}
\nonumber
\Ll_A \alpha^k &\stackrel{(\ref{eq:LK-deriv})}= e^i \wedge (\imath_{A e_i} d\alpha^k) - d(e^i \wedge (\imath_{A e_i} \alpha^k))\\
\label{eq:der-endo}
&= A \cdot d\alpha^k - d (A \cdot \alpha^k),
\end{align}
where we denote by $\cdot$ the pointwise action of $A_p \in End(T_pM)$ on $\Lambda^k T_p^*M$. Observe that by (\ref{eq:FN-homom}) we have $\Ll_{[A,A]^{FN}} = 2 (\Ll_A)^2$, so that the derivation $\Ll_A: \Om^k(M) \to \Om^{k+1}(M)$ is a differential iff $[A,A]_N = 0$.

For instance, if $A = Id$ then $I \cdot \alpha^k = e^i \wedge (\imath_{e_i} \alpha^k) = k \alpha^k$, so that
\[
\Ll_I \alpha^k \stackrel{(\ref{eq:der-endo})}= I \cdot d\alpha^k - d (I \cdot \alpha^k) = (k+1) d\alpha^k - d(k\alpha^k) = d\alpha^k.
\]
To see another example, let $A = J$ be an almost complex structure. Then $[J,J]^{FN} = 2 [J,J]_N = 0$ iff $J$ is integrable, and in this case one calculates from (\ref{eq:der-endo}) that $\Ll_J = -d^c = i(\p - \bar{\p})$ is the negative of the complex differential, where $d = \p + \bar{\p}$ is the decomposition into the holomorphic and anti-holomorphic part of $d$. In particular, $H^*_J(\Om^*(M)) \cong H^*_{dR}(M)$ coincides with the deRham cohomology.
\end{example}

We end this section by providing a formula for the Fr\"olicher-Nijenhuis bracket for those types of forms which we shall be concerned with. Recall from the introduction that on a Riemannian manifold $(M, g)$ we define the map
\begin{equation} \label{eq:def-partial}
\delta = \delta_g: \Lambda^k V^* \longrightarrow \Lambda^{k-1} V^* \otimes V, \qquad \delta_g(\alpha^k) := (\imath_{e_i} \alpha^k) \otimes (e^i)^\#,
\end{equation}
taking the sum over some basis $(e_i)$ of $T_pM$ with dual basis $(e^i)$ of $T^\ast_pM$. This implies that to each $\Psi \in \Om^{k+1}(M)$ we may associate a section $\delta_g(\Psi) \in \Om^k(M, TM)$.

\begin{proposition} \label{prop:FN}
Let $(M, g)$ be an $n$-dimensional Riemannian manifold of dimension $n$ and let $\Psi_l \in \Om^{k_l+1}(M)$, $l = 1,2$. Moreover, let
\[
K_l := \delta_g(\Psi_l) \in \Om^{k_l}(M, TM)
\]
with the map $\delta_g$ from (\ref{eq:def-partial}).

Then the Fr\"olicher-Nijenhuis bracket at $p \in M$ is given as
\begin{align*}
{}[K_1, K_2&]^{FN}_p = \Big( (\imath_{e_i} \Psi_1) \wedge (\imath_{e_j} \nabla_{e_i} \Psi_2) - (-1)^{k_1} (\imath_{e_j} \imath_{e_i} \Psi_1) \wedge e^k \wedge \imath_{e_i} \nabla_{e_k} \Psi_2\\
- &(\imath_{e_j} \nabla_{e_i} \Psi_1) \wedge (\imath_{e_i} \Psi_2) - (-1)^{k_1} e^k \wedge \imath_{e_i} \nabla_{e_k} \Psi_1 \wedge (\imath_{e_j} \imath_{e_i} \Psi_2)\Big) \otimes (e^j)^\#,
\end{align*}
where $(e_i)$ is an arbitrary basis of $T_pM$ with dual basis $(e^i)$ of $T_p^*M$.
In particular, if $K_1 = K_2 =: K$ and $k_1 = k_2$ is odd, then
\begin{align*}
{}[K, K]^{FN}_p = & 2 \Big( (\imath_{e_i} \Psi) \wedge (\imath_{e_j} \nabla_{e_i} \Psi) + (\imath_{e_j}\imath_{e_i} \Psi) \wedge e^k \wedge \imath_{e_i} \nabla_{e_k} \Psi \Big) \otimes (e^j)^\#.
\end{align*}
\end{proposition}

\begin{remark} If $K_1 = K_2 = K$ and $k_1 = k_2$ is even, then $[K, K]^{FN} = 0$ due to the graded skew symmetry of the bracket. Furthermore, observe that $(e^j)^\# = e_j$ in case $(e_i)$ is an orthonormal basis.
\end{remark}

\begin{proof}
Evidently, if this formula holds for {\em some }basis $(e_j)$ with dual basis $(e^j)$, then it holds for {\em any }basis. Therefore, it suffices to show the assertion for an orthonormal basis $(e_j)$ in which case $(e^j)^\# = e_j$.

Choose geodesic normal coordinates $(x^i)$ around $p \in M$ in such a way that $(\p_i)_p := (\p / \p x^i)_p$ is an orthonormal basis of $T_pM$.
The dual basis of $\p_i$ is $dx^i$, whence $(dx^i)^\# = g^{ij} \p_j$. Thus,
\[
K_l = (\imath_{\p_i} \Psi_l) \otimes (dx^i)^\# = g^{ij} (\imath_{\p_i} \Psi_l) \otimes \p_j.
\]
Thus, by (\ref{eq:kms})
\begin{align*}
{}[K_1, K_2]^{FN} = 
&[g^{ij} (\imath_{\p_i} \Psi_1) \otimes \p_j, g^{rs} (\imath_{\p_r} \Psi_2) \otimes \p_s]^{FN}\\
= 
&\left(
g^{ij} (\imath_{\p_i} \Psi_1) \wedge \Ll_{\p_j} (g^{rs} (\imath_{\p_r} \Psi_2)) \otimes \p_s \right.\\
&-
\Ll_{\p_s} (g^{ij} (\imath_{\p_i} \Psi_1)) \wedge g^{rs} (\imath_{\p_r} \Psi_2) \otimes \p_j \\
&+
(-1)^{k_1} d (g^{ij} (\imath_{\p_i} \Psi_1)) \wedge \imath_{\p_j} (g^{rs} (\imath_{\p_r} \Psi_2)) \otimes \p_s \\
&+(-1)^{k_1}(\imath_{\p_s}(g^{ij} (\imath_{\p_i} \Psi_1)) \wedge d(g^{rs} (\imath_{\p_r} \Psi_2)) \otimes \p_j.
\end{align*}
Since at $p$, $g_{ij} = g^{ij} = \delta_{ij}, \p_r g_{ij} = 0$,
$\Ll_{\p_j} \Psi = \nabla_{e_j} \Psi$, $\nabla_{\p_i} \p_j = 0$, and $\p_j = (e^j)^\#$, the asserted formula follows.
\end{proof}

%%%%%%%%%%%%%%%%%%%%%%%%%%%%%%%%%%%%%%%%%%%%%%%%%%%%%%%%%
\section{Cross products and $G_2$-structures}\label{sec:g2}
%%%%%%%%%%%%%%%%%%%%%%%%%%%%%%%%%%%%%%%%%%%%%%%%%%%%%%%%%

%%%%%%%%%%%%%%%%%%%%%%%%%%%%%%%%%%%%%%%%%%%%%%%%%%%%%%%%%
\subsection{$G_2$-structures and associated cross products}\label{subs:g2str}
%%%%%%%%%%%%%%%%%%%%%%%%%%%%%%%%%%%%%%%%%%%%%%%%%%%%%%%%%

In this section we collect some basic facts on $G_2$-structures, see e.g. \cite{Humphreys}, \cite{Bryant1987}, \cite{FG1982}, \cite{HL1982} for references.

Let $M$ be an oriented 7-manifold. A {\it $G_2$-structure on $M$ }is a $3$-form $\varphi \in \Om^3(M)$ such that at each $p \in M$ there is a positively oriented basis $(e_i)$ of $T_pM$ with dual basis $(e^i)$ such that 
\begin{equation} \label{varphi}
\varphi_p = e^{123} + e^{145} + e^{167} + e^{246} - e^{257} - e^{347} - e^{356},
\end{equation}
where $e^{i_1 \dots i_k}$ is short for $e^{i_1} \wedge \cdots \wedge e^{i_k}$. We call such a basis a {\em $G_2$-frame}. The stabilizer of $\varphi_p$ is isomorphic to the exceptional group $G_2$, and there is a unique $G_2$-invariant Riemannian metric $g_\varphi$ on $M$ such that each $G_2$-frame is orthonormal. In particular, the Hodge-dual of $\varphi$ w.r.t. $g_\varphi$ is given by 
\begin{equation} \label{varphi*}
\ast_{g_\varphi} \varphi = e^{4567} + e^{2367} + e^{2345} + e^{1357} - e^{1346} - e^{1256} - e^{1247}.
\end{equation}

The set of $G_2$-frames yields a principal $G_2$-bundle
\[
G_2(M) = G_2(M, \varphi) \longrightarrow M,
\]
whence for each $G_2$-module $V$ we denote by

\begin{equation} \label{eq:def-G2-bundles}
V(M) := G_2(M) \times_{G_2} V \longrightarrow M
\end{equation}
the associated vector bundle over $M$. For instance,
\[
V_7(M) \cong TM \cong T^*M.
\]

\begin{definition} \label{def:Crchi} \cite{HL1982}
Let $(M, \varphi)$ be an oriented manifold with a $G_2$-structure $\varphi$ and the induced Riemannian metric $g = g_\varphi$. Then the $TM$-valued forms $Cr_\varphi \in \Om^2(M, TM)$ and $\chi_\varphi \in \Om^3(M, TM)$ are defined by 
\[
Cr_\varphi := \delta_{g_\varphi}(\varphi) \qquad \mbox{and} \qquad \chi_\varphi := -\delta_g(\ast \varphi),
\]
and are called the {\em $2$-fold and $3$-fold cross product on $M$}, respectively. That is, for $x,y,z,w \in TM$ we have
\begin{align*}
g_\varphi(Cr_\varphi (x,y), z) = \varphi (x,y,z), \qquad
g_\varphi(\chi_\varphi (x,y,z), w) = * \varphi (x,y,z,w).
\end{align*} 
\end{definition}

We shall usually suppress the indices $\varphi$ for $g, Cr$ and $\chi$ if it is clear from the context which $G_2$-structure $\varphi$ is used.

\begin{remark}
If we use the $G_2$-structure to identify each $T_pM \cong Im \O$ with the imaginary octonians, then $Cr$ and $\chi$ can be interpreted w.r.t. the octonian product $\cdot: \O \times \O \to \O$ as
\[
Cr(x, y) := (x \cdot y)_{Im \O} \qquad \mbox{and}\qquad \chi(x,y,z) := ((x \cdot y) \cdot z - x \cdot (y \cdot z))_{Im \O}.
\]
\end{remark}

We summarize important known facts about the decomposition of tensor products of $G_2$-modules into irreducible summands which are well known, see e.g. \cite[Section 2]{Kar2005}. We denote by $V_k$ the $k$-dimensional irreducible $G_2$-module if there is a unique such module. For instance, $V_7$ is the irreducible $7$-dimensional $G_2$-module from above, and $V_7^* \cong V_7$. For its exterior powers, we obtain the decompositions
\begin{equation} \label{eq:DiffForm-V7}
\begin{array}{rlrl}
\Lambda^0 V_7 \cong \Lambda^7 V_7 \cong V_1, \qquad
& \Lambda^2 V_7 \cong \Lambda^5 V_7 \cong V_7 \oplus V_{14},\\[2mm]
\Lambda^1 V_7 \cong \Lambda^6 V_7 \cong V_7, \qquad
& \Lambda^3 V_7 \cong \Lambda^4 V_7 \cong V_1 \oplus V_7 \oplus V_{27},
\end{array}
\end{equation}
where $\Lambda^k V_7 \cong \Lambda^{7-k} V_7$ due to the Hodge isomorphism. We denote by $\Lambda^k_l V_7 \subset \Lambda^k V_7$ the subspace isomorphic to $V_l$ in the above notation. Evidently, $\Lambda^3_1 V_7$ and $\Lambda^4_1 V_7$ are spanned by $\varphi$ and $\ast \varphi$, respectively. For the decompositions of $\Lambda^2 V_7$ and $\Lambda^5 V_7$ the following descriptions are well known.

\begin{align}
\nonumber
\Lambda^2_7 V_7 &= \{ \imath_v \varphi \mid v \in V_7\},\\
\nonumber
\Lambda^2_{14} V_7 &= \{ \alpha^2 \in \Lambda^2 V_7 \mid \ast \varphi \wedge \alpha^2 = 0\},\\
\nonumber
\Lambda^3_1 V_7 &= \R \varphi,\\
\label{decom-L-V7}
\Lambda^3_7 V_7 &= \{ \imath_v \ast \varphi \mid v \in V_7\},\\
\nonumber
\Lambda^4_1 V_7 &= \R \ast \varphi,\\
\nonumber
\Lambda^4_7 V_7 &= \varphi \wedge V_7 = \{ \varphi \wedge v \mid v \in V_7\},\\
\nonumber
\Lambda^5_7 V_7 &= \ast \varphi \wedge V_7 = \{ \ast \varphi \wedge v \mid v \in V_7\},\\
\nonumber
\Lambda^5_{14} V_7 &= \{ \alpha^5 \in \Lambda^5 V_7 \mid \alpha^5 \wedge (\imath_v \varphi) = 0 \; \mbox{for all $v \in V_7$}\}.
\end{align}

%We also decompose the symmetric powers $\odot^k V_7$ as
%\begin{equation} \label{decom-Sk-V7}
%\odot^2 V_7 \cong V_1 \oplus V_{27}, \qquad \odot^3 V_7 \cong V_7 \oplus V_{77}.
%\end{equation}
%Here, the $1$-dimensional subspace in $\odot^2 V_7$ is given by the metric $g_\varphi$, and $V_{27}$ corresponds to the trace free symmetric matrices; the $7$-dimensional subspace in $\odot^3 V_7$ is given by multiplication with the metric $g_\varphi$, whereas $V_{77}$ consists of harmonic polynomials of degree $3$ in $7$ variables; see \cite{Humphreys}.

We also point out that all representations of $G_2$ are of real type, meaning that for any real irreducible representation $V$ of $G_2$ the complexified space $V^\C := V \otimes \C$ is (complex) irreducible; equivalently, a real irreducible representation of $G_2$ does not admit a $G_2$-invariant complex structure.

These decompositions are used in the appendix to obtain many formulas which will be used in the sequel.

%%%%%%%%%%%%%%%%%%%%%%%%%%%%%%%%%%%%%%%%%%%%%%%%%%%%%%%%%
\subsection{The torsion of manifolds with a $G_2$-structure}
%%%%%%%%%%%%%%%%%%%%%%%%%%%%%%%%%%%%%%%%%%%%%%%%%%%%%%%%%

Let $(M, \varphi)$ be a manifold with a $G_2$-structure with the corresponding Riemannian metric $g = g_\varphi$, and let $\nabla$ be the Levi-Civita connection of $g$. In general, $\varphi$ and $\ast \varphi$ will not be parallel w.r.t. $\nabla$, and the failure of their parallelity can be described in the following way which is essentially a reformulation of the intrinsic torsion of a $G_2$-structure discussed in \cite{Bryant2005} and \cite{Kar2005}.

\begin{proposition} \label{prop:torsionG2}
Let $(M, \varphi)$ be a manifold with a $G_2$-structure with associated Riemannian metric $g = g_\varphi$ and Levi-Civita connection $\nabla$. Then there is a section $T \in \Om^1(M, TM) = \Gamma(End(TM))$ such that for all $v \in TM$ we have
\begin{equation} \label{eq:nabla-G2}
\nabla_v \varphi = \imath_{T(v)} \ast \varphi \qquad \mbox{and} \qquad \nabla_v * \varphi = -(T(v))^\flat \wedge \varphi.
\end{equation}
\end{proposition}

%\begin{proof}
%For any curve $c:(-\eps, \eps) \to M$, the parallel translation $P_t$ along $c$ of $\varphi_{c(t)}$ to $p = c(0)$ lies in the $SO(T_pM, g_p)$-orbit of $\varphi_p$, i.e., there is a curve $g_t \in SO(T_pM, g_p)$ such that $P_t(\varphi_{c(t)}) = g_t \cdot \varphi_p$ and $g_e = 1$. Thus, taking the derivative and setting $v := \dot c(0)$, we have
%\[
%\nabla_v \varphi = \dot g_0 \cdot \varphi_p \in \so(T_pM, g_p) \cdot \varphi_p,
%\]
%and since $\so(T_pM, g_p) = \g_2 \oplus \g_2^\perp$, where $\g_2$ is the Lie algebra of the group $G_2 \subset SO(T_pM, g_p)$ stabilizing $\varphi_p$, we have
%\[
%\so(T_pM, g_p) \cdot \varphi_p = \g_2^\perp \cdot \varphi_p = \Lambda^3_7 T_p^\ast M,
%\]
%since $\g_2^\perp \cong V_7$ as a $G_2$-module, and since by (\ref{eq:DiffForm-V7}) there is a unique submodule of $\Lambda^3 T_p^\ast M$ isomorphic to $V_7$. Thus, $\nabla_v \varphi \in \Lambda^3_7 T_p^\ast M$, whence (\ref{decom-L-V7}) implies that $\nabla_v \varphi = \imath_{T(v)} \ast \varphi$ for some unique element $T(v) \in T_pM$, showing the first equation in (\ref{eq:nabla-G2}). The second equation follows, as
%\[
%\nabla_v \ast \varphi = \ast \nabla_v\varphi = \ast \imath_{T(v)} \ast \varphi = -\ast^2 (T(v))^\flat \wedge \varphi = -(T(v))^\flat \wedge \varphi.
%\]
%
%\vspace{-7.3mm}
%\end{proof}

Thus, the section $T \in \Om^1(M, TM)$ measures how $\varphi$ fails to be parallel, and this has been described in Fern\'{a}ndez and Gray \cite{FG1982} by slightly different means. In fact, it contains the same information as the intrinsic torsion of the $G_2$-structure in the sense of \cite{Bryant2005}, whence we use the following terminology.

\begin{definition} \label{def:torsionG2}
Let $(M, \varphi)$ be a manifold with a $G_2$-structure. The section $T \in \Om^1(M, TM)$ for which (\ref{eq:nabla-G2}) holds is called the {\em torsion endomorphism }of the $G_2$-structure.
\end{definition}

For an orthonormal frame $(e_i)$ of $T_pM$ we define the coefficients of $T$ by
\begin{equation} \label{def-tij}
t_{ij} := \la T(e_i), e_j \ra, \qquad \mbox{so that} \qquad T(e_i) = t_{ij} e_j.
\end{equation}
Furthermore, we define the form
\begin{equation} \label{def:tau}
\tau := t_{ij} e^{ij} = \dfrac12 (t_{ij} - t_{ji}) e^{ij} = e^i \wedge T(e_i)^\flat \in \Lambda^2 V_7^*.
\end{equation}
For the exterior derivatives of $\varphi$ and $\ast \varphi$, we have 
\begin{equation}
\label{d*-T}
d\varphi_p = T_p^\top(e_i)^\flat \wedge (\imath_{e_i} \ast \varphi_p) \qquad \mbox{and} \qquad
d\ast \varphi_p = -\tau_p \wedge \varphi_p,
\end{equation}
where we sum over an orthonormal basis $(e_i)$ of $T_pM$ in the first equation and where $T_p^\top$ denotes the transpose matrix of $T_p$. %Indeed, picking an orthonormal local frame $(e_i)$ which is normal at $p \in M$ we have
%\begin{align*}
%d\varphi_p & = e^k \wedge \nabla_{e_k} \varphi|_p \stackrel{(\ref{eq:nabla-G2})} = e^k \wedge \imath_{T_p(e_k)} \ast \varphi = T_p^\top(e_l)^\flat \wedge (\imath_{e_l} \ast \varphi_p)\\
%d\ast \varphi_p & = e^k \wedge \nabla_{e_k} \ast \varphi|_p \stackrel{(\ref{eq:nabla-G2})} = -e^k \wedge T_p(e_k)^\flat \wedge \varphi_p = -\tau_p \wedge \varphi_p
%\end{align*}
%from which (\ref{d*-T}) follows.
In particular, it is now a straightforward calculation to show that $(M, \varphi)$ is torsion free at $p \in M$ (i.e., $T_p = 0$) iff $d\varphi_p = 0$ and $d\ast \varphi_p = 0$ (cf. \cite{FG1982}).

%%%%%%%%%%%%%%%%%%%%%%%%%%%%%%%%%%%%%%%%%%%%%%%%%%%%%%%%%
\subsection{The Fr\"olicher-Nijenhuis brackets on a manifold with a $G_2$-structure}
%%%%%%%%%%%%%%%%%%%%%%%%%%%%%%%%%%%%%%%%%%%%%%%%%%%%%%%%%
In this section, we shall compute part of their Fr\"olicher-Nijenhuis brackets of the sections $Cr = \delta_g \varphi \in \Om^2(M, TM)$ and $\chi = -\delta_g \ast \varphi \in \Om^3(M, TM)$. from Definition \ref{def:Crchi} on a manifold $M$ with a $G_2$-structure $\varphi$.

The Fr\"olicher-Nijenhuis bracket $[Cr, Cr]^{FN}$ vanishes identically due to the graded skew-symmetry of the bracket. On the other hand, the Fr\"olicher-Nijenhuis brackets $[Cr, \chi]^{FN}$ and $[\chi, \chi]^{FN}$ are elements of $\Om^5(M, TM)$ and $\Om^6(M, TM)$, respectively.

Due to the decomposition $\Lambda^5 V_7 = \Lambda^5_7 V_7 \oplus \Lambda^5_{14} V_7$ as a $G_2$-module, we may decompose
\[
\Om^5(M, TM) = \Gamma(M, \Lambda^5_7 T^*M \otimes TM) \oplus \Gamma(M, \Lambda^5_{14} T^*M \otimes TM),
\]
and we denote the projections onto the two summands by $\pi_7$ and $\pi_{14}$, respectively. We now wish to show Theorem \ref{thm:Brackets-intro} from the introduction. In order to work towards the proof, we first calculate $\pi_7([Cr, \chi]^{FN})$.

\begin{proposition}\label{prop:Cr-chi}
Let $(M, g, \varphi)$ be a manifold with a $G_2$-structure and let $T \in \Om^1(M, TM)$ be its torsion endomorphism. Then for each $p \in $M,
\begin{equation} \label{eq:pi-Crchi}
\pi_7([Cr, \chi]_p^{FN}) = 2 \ast \varphi \wedge \left( \left(T_p^\top - 2 T_p - tr(T_p)\right) e_i\right)^\flat \otimes e_i,
\end{equation}
summing over an orthonormal basis $(e_i)$ of $T_pM$, where $T^\top$ denotes the transpose of $T$. In particular, $\pi_7([Cr, \chi]^{FN}_p) = 0$ if and only if $T_p = 0$ if and only if $[Cr, \chi]^{FN}_p = 0$.
\end{proposition}

\begin{proof}
We fix $p \in M$ and use normal coordinates around $p$. Then in order to calculate $[Cr, \chi]^{FN}_p$ we apply Proposition \ref{prop:FN} to $\Psi_1 = \varphi$ and $\Psi_2 = \ast \varphi$ and obtain
\[
{}-[Cr, \chi]^{FN}_p = [(\imath_{e_i} \varphi) \otimes e_i, (\imath_{e_j} \ast \varphi) \otimes e_j]^{FN}_p =: \beta_j \otimes e_j,
\]
where
\begin{align}
\nonumber
\beta_j =&\; (\imath_{e_k} \varphi) \wedge (\imath_{e_j} \nabla_{e_k} \ast \varphi) - (\imath_{e_j}\imath_{e_k} \varphi) \wedge e^l \wedge (\imath_{e_k} \nabla_{e_l} \ast \varphi)\\
\label{Cr-chiVer1}
&-(\imath_{e_j} \nabla_{e_k} \varphi) \wedge (\imath_{e_k} \ast \varphi) - e^l \wedge (\imath_{e_k} \nabla_{e_l} \varphi) \wedge (\imath_{e_j} \imath_{e_k} \ast \varphi).
\end{align}

Decomposing $\Lambda^5 T_p^*M$ according to (\ref{decom-L-V7}), we write $\beta_j = \ast \varphi \wedge v_j^\flat + \beta_j^{14}$ with $v_j \in T_pM$ and $\beta_j^{14} \in \Lambda^5_{14} T_p^*M$, so that $\pi_7([Cr, \chi]^{FN}) = \ast \varphi \wedge v_j^\flat \otimes e_j$. Let
\begin{equation} \label{def-bij}
b_{ij} = \ast ((\imath_{e_i} \varphi) \wedge \beta_j).
\end{equation}
Then as $(\imath_{e_i} \varphi) \wedge \beta_j^{14} = 0$ by (\ref{decom-L-V7}) and $(\imath_{e_i} \varphi) \wedge \ast \varphi \wedge v_j^\flat = 3 \la e_i, v_j\ra \vol$ by (\ref{eq:form-0}), it follows that
\begin{equation} \label{def-pi7}
\pi_7(-[Cr, \chi]^{FN}_p) = \dfrac13 b_{ij} (\ast \varphi \wedge e^i) \otimes e_j.
\end{equation}

In order to determine the coefficients $b_{ij}$, we decompose $\beta_j$ into the four summands from (\ref{Cr-chiVer1}). Then from the first summand we get

\begin{align}
\nonumber
(\imath_{e_i} \varphi) \wedge (\imath_{e_k} \varphi) \wedge (\imath_{e_j} \nabla_{e_k} \ast \varphi)\stackrel{(\ref{eq:nabla-G2})} =&\; - (\imath_{e_i} \varphi) \wedge (\imath_{e_k} \varphi) \wedge (\imath_{e_j} (T(e_k)^\flat \wedge \varphi))\\
\nonumber
=&\; - (\imath_{e_i} \varphi) \wedge (\imath_{e_k} \varphi) \wedge (t_{kj} \varphi - T(e_k)^\flat \wedge (\imath_{e_j} \varphi))\\
\nonumber
\stackrel{(\ref{eq:form-1}),(\ref{eq:form-2})} =&\; - 6 t_{kj} \delta_{ik} \vol + 2 (\delta_{ik} t_{kj} + t_{ki} \delta_{kj} + t_{kk} \delta_{ij}) \vol\\
\label{b_ij-1}
=&\; 2 (t_{ji} - 2 t_{ij} + tr(T) \delta_{ij}) \vol.
\end{align}

From the second summand we obtain

\begin{align}
\nonumber
-(\imath_{e_i} \varphi) \wedge (\imath_{e_j} \imath_{e_k} \varphi) \wedge e^l \wedge &(\imath_{e_k} \nabla_{e_l} \ast \varphi)\\
\nonumber
\stackrel{(\ref{eq:nabla-G2})} =&\; (\imath_{e_i} \varphi) \wedge (\imath_{e_j}\imath_{e_k} \varphi) \wedge e^l \wedge (\imath_{e_k} (T(e_l)^\flat \wedge \varphi))\\
\nonumber
 =&\; (\imath_{e_i} \varphi) \wedge (\imath_{e_j}\imath_{e_k} \varphi) \wedge e^l \wedge (t_{lk} \varphi - T(e_l)^\flat \wedge (\imath_{e_k} \varphi))\\
\nonumber
\stackrel{(\ref{eq:form-3})} =&\; t_{lk} \Big(2(\delta_{kl} \delta_{ji} - \delta_{jl} \delta_{ki}) \vol - 2 e^{jkli} \wedge \varphi\Big)\\
\nonumber
& - (\imath_{e_i} \varphi) \wedge (\imath_{e_j}\imath_{e_k} \varphi) \wedge e^l \wedge T(e_l)^\flat \wedge (\imath_{e_k} \varphi)\\
\nonumber
=&\; 2(tr(T) \delta_{ij} - t_{ji}) \vol - 2 e^{ij} \wedge \tau \wedge \varphi\\
\nonumber
& - \frac12 (\imath_{e_i} \varphi) \wedge \imath_{e_j}((\imath_{e_k} \varphi) \wedge (\imath_{e_k} \varphi)) \wedge \tau\\
\nonumber
\stackrel{(\ref{formulas2})}=&\; 2(tr(T) \delta_{ij} - t_{ji}) \vol - 2 e^{ij} \wedge \tau \wedge \varphi\\
\nonumber
& - 3 (\imath_{e_i} \varphi) \wedge (\imath_{e_j}\ast \varphi) \wedge \tau\\
\nonumber
\stackrel{(\ref{eq:form13})}=&\; 2(tr(T) \delta_{ij} - t_{ji}) \vol - 2 e^{ij} \wedge \tau \wedge \varphi\\
\nonumber
& -3 (-2 \ast e^{ji} + e^{ji} \wedge \varphi) \wedge \tau\\
\nonumber
=&\; 2(tr(T) \delta_{ij} - t_{ji}) \vol + e^{ij} \wedge \tau \wedge \varphi\\
\nonumber
& + 6 t_{kl}(\delta_{jk} \delta_{il} - \delta_{jl} \delta_{ik}) \vol\\
\label{b_ij-2}
=&\; 2(tr(T) \delta_{ij} + 2 t_{ji} - 3 t_{ij}) \vol + e^{ij} \wedge \tau \wedge \varphi.
\end{align}

From the third term in (\ref{Cr-chiVer1}) we get

\begin{align}
\nonumber
-(\imath_{e_i} \varphi) \wedge (\imath_{e_j} \nabla_{e_k} \varphi) \wedge& (\imath_{e_k} \ast \varphi)\\
\nonumber
\stackrel{(\ref{eq:nabla-G2})}=&\; -(\imath_{e_i} \varphi) \wedge (\imath_{e_j} \imath_{T(e_k)} \ast\varphi) \wedge (\imath_{e_k} \ast \varphi)\\
\nonumber
\stackrel{(\ref{eq:form-4})}=&\; -2 (t_{kk} \delta_{ji} - \delta_{jk} t_{ki}) \vol - e^j \wedge T(e_k)^\flat \wedge e^{ki} \wedge \varphi\\
\label{b_ij-3}
=&\; -2 (tr(T) \delta_{ij} - t_{ji}) \vol - e^{ij} \wedge \tau \wedge \varphi.
\end{align}

Finally, from the last term in (\ref{Cr-chiVer1}) we get

\begin{align}
\nonumber
-(\imath_{e_i} \varphi) \wedge e^l \wedge (\imath_{e_k} \nabla_{e_l} \varphi) \wedge& (\imath_{e_j} \imath_{e_k} \ast \varphi)\\
\nonumber
\stackrel{(\ref{eq:nabla-G2})}=&\; -(\imath_{e_i} \varphi) \wedge e^l \wedge (\imath_{e_k} \imath_{T(e_l)} \ast \varphi) \wedge (\imath_{e_j} \imath_{e_k} \ast \varphi)\\
\nonumber
\stackrel{(\ref{formulas-3})}=&\; 2 (\imath_{e_i} \varphi) \wedge e^l \wedge (\imath_{T(e_l)} \varphi) \wedge (\imath_{e_j} \varphi)\\
\nonumber
\stackrel{(\ref{eq:form-2})}=&\; 4 (t_{li} \delta_{lj} + \delta_{li} t_{lj} + t_{ll} \delta_{ij}) \vol\\
\label{b_ij-4}
=&\; 4 (t_{ji} + t_{ij} + tr(T) \delta_{ij}) \vol.
\end{align}

Thus, adding (\ref{b_ij-1}) through (\ref{b_ij-4}), we get from (\ref{Cr-chiVer1}) that

\begin{align*}
b_{ij} \vol = (\imath_{e_i} \varphi) \wedge \beta_j =&\; 2 (t_{ji} - 2 t_{ij} + tr(T) \delta_{ij}) \vol\\
&+ 2(tr(T) \delta_{ij} + 2 t_{ji} - 3 t_{ij}) \vol + e^{ij} \wedge \tau \wedge \varphi\\
&- 2 (tr(T) \delta_{ij} - t_{ji}) \vol - e^{ij} \wedge \tau \wedge \varphi\\
&+ 4 (t_{ji} + t_{ij} + tr(T) \delta_{ij}) \vol\\
=&\; 6 (2 t_{ji} - t_{ij} + tr(T) \delta_{ij}) \vol,
\end{align*}
and hence, (\ref{def-pi7}) implies (\ref{eq:pi-Crchi}).

Thus, $\pi_7([Cr, \chi]^{FN}_p) = 0$ iff $T_p^\top - 2 T_p - tr(T_p)Id = 0$. Taking the trace, this implies that $tr(T_p) - 2 tr(T_p) - 7 tr(T_p) = 0$ and hence, $tr(T_p) = 0$, and $T_p^\top - 2 T_p = 0$ evidently implies that $T_p = 0$. That is, $\pi_7([Cr, \chi]^{FN}_p) = 0$ iff $T_p = 0$, showing the last statement.
\end{proof}

Next, let us consider the bracket $[\chi, \chi]^{FN}$.

\begin{proposition}\label{prop:chi-chi}
Let $(M, g, \varphi)$ be a manifold with a $G_2$-structure and let $T \in \Om^1(M, TM)$ be its torsion endomorphism with the associated form $\tau \in \Om^2(M)$ from (\ref{def:tau}). Then for each $p \in $M,
\begin{equation} \label{eq:pi-chichi}
{}[\chi, \chi]_p^{FN} = -4 \ast (T_p + T_p^\top)(e_i) \otimes e_i + 6 e^i \wedge \tau_p \wedge \varphi \otimes e_i,
\end{equation}
summing over an orthonormal basis $(e_i)$ of $T_pM$. In particular, $[\chi, \chi]_p^{FN} = 0$ if and only if $T_p = 0$.
\end{proposition}

\begin{proof}
According to Proposition \ref{prop:FN} we have
\[
{}[\chi, \chi]^{FN}_p = \gamma_j \otimes e_j,
\]
where
\begin{equation} \label{chichiVer1}
\gamma_j = 2 ((\imath_{e_k} \ast \varphi) \wedge (\imath_{e_j} \nabla_{e_k} \ast \varphi) + (\imath_{e_j} \imath_{e_k} \ast \varphi) \wedge e^l \wedge (\imath_{e_k} \nabla_{e_l} \ast \varphi)).
\end{equation}

Now let $c_{ij} := \ast ( e^i \wedge \gamma_j ) = \la e_i, (\ast \gamma_j)^\flat\ra$. Then 
\begin{equation} \label{def-pi7chichi}
{}[\chi, \chi]^{FN}_p = c_{ij} \ast e^i \otimes e_j,
\end{equation}

In order to evaluate the coefficients $c_{ij}$, we consider the two summands in (\ref{chichiVer1}) separately, and obtain from the first one

\begin{align}
\nonumber
e^i \wedge (\imath_{e_k} \ast \varphi) &\wedge (\imath_{e_j} \nabla_{e_k} \ast \varphi) \stackrel{(\ref{eq:nabla-G2})}=\; - e^i \wedge (\imath_{e_k} \ast \varphi) \wedge (\imath_{e_j} (T(e_k)^\flat \wedge \varphi))\\
\nonumber
=&\; - t_{kj} e^i \wedge (\imath_{e_k} \ast \varphi) \wedge \varphi + e^i \wedge (\imath_{e_k} \ast \varphi) \wedge T(e_k)^\flat \wedge (\imath_{e_j} \varphi)\\
\nonumber
\stackrel{(\ref{eq:form-11}), (\ref{eq:form-5})}=&\; -4 t_{kj} \delta_{ik} \vol + 2(\delta_{ik} t_{kj} - t_{kk} \delta_{ij}) \vol + T(e_k)^\flat \wedge e^{ikj} \wedge \varphi\\
\label{c_ij-1}=&\; -2(t_{ij} + tr(T) \delta_{ij}) \vol + e^{ij} \wedge \tau \wedge \varphi .
\end{align}

From the second summand in (\ref{chichiVer1}) we calculate

\begin{align}
\nonumber
e^i \wedge (\imath_{e_j} & \imath_{e_k} \ast \varphi) \wedge e^l \wedge (\imath_{e_k} \nabla_{e_l} \ast \varphi)\\
\nonumber
\stackrel{(\ref{eq:nabla-G2})}=&\; - e^{il} \wedge (\imath_{e_j} \imath_{e_k} \ast \varphi) \wedge (\imath_{e_k} (T(e_l)^\flat \wedge \varphi))\\
\nonumber
=&\; - t_{lk} e^{il} \wedge (\imath_{e_j} \imath_{e_k} \ast \varphi) \wedge \varphi + e^i \wedge (\imath_{e_j} \imath_{e_k} \ast \varphi) \wedge \tau \wedge (\imath_{e_k} \varphi)\\
\nonumber
\stackrel{(\ref{eq:form-6}), (\ref{formulas-4})}=&\; - t_{lk} \big(2 (\delta_{lj} \delta_{ik} - \delta_{ij} \delta_{lk}) \vol - e^{iljk} \wedge \varphi\big) + 3 e^{ij} \wedge \tau \wedge \varphi\\
\label{c_ij-2}=&\; - 2 (t_{ji} - tr(T) \delta_{ij}) \vol + 2 e^{ij} \wedge \tau \wedge \varphi.
\end{align}

Thus, adding (\ref{c_ij-1}) and (\ref{c_ij-2}), equation (\ref{chichiVer1}) yields

\begin{align*}
c_{ij} \vol = e^i \wedge \gamma_j =&\; 2 \Big( -2(t_{ij} + tr(T) \delta_{ij}) \vol + e^{ij} \wedge \tau \wedge \varphi\\
&\qquad - 2 (t_{ji} - tr(T) \delta_{ij}) \vol + 2 e^{ij} \wedge \tau \wedge \varphi\Big)\\
=&\; 2 (-2(t_{ij} + t_{ji}) + 3 \la \ast e^i, e^j \wedge \tau \wedge \varphi \ra) \vol,
\end{align*}
and from this and (\ref{def-pi7chichi}), the formula (\ref{eq:pi-chichi}) follows.

In order to show the last statement, observe that $[\chi, \chi]^{FN}_p=0$ iff $c_{ij} = 0$ for all $i, j$. Since then $c_{ij} + c_{ji} = -8(t_{ij} + t_{ji})$, it follows that $t_{ij} + t_{ji} = 0$ and hence 
$0 = c_{ij} = 6 \ast ( e^{ij} \wedge \tau \wedge \varphi)$ for all $i, j$ which implies that $\tau = t_{kl}e^{kl} = 0$ and hence, $t_{kl}= t_{lk}$. All of this together implies that $t_{ij} = 0$ for all $i, j$, and hence, $T_p=0$ as asserted.
\end{proof}

We are now ready to show the following result which immediately implies Theorem \ref{thm:Brackets-intro} from the introduction.

\begin{theorem} \label{thm:Brackets}
Let $(M^7, \varphi)$ be a manifold with a $G_2$-structure with associated metric $g = g_\varphi$, let $\nabla$ be the Levi-Civita connection of $g$, and let $T \in \Om^1(M^7, TM^7)$ be its torsion endomorphism defined in Definition \ref{def:torsionG2}. Then for every $p \in M^7$ the following are equivalent.
\begin{enumerate}
\item
$T_p = 0 \in T_p^*M^7 \otimes T_pM^7$.
\item
The $G_2$-structure is torsion-free at $p$, i.e., $(\nabla \varphi)_p = 0$.
\item
$\pi_7([Cr, \chi]_p^{FN}) = 0 \in \Lambda^5_7 T_p^*M^7 \otimes T_pM^7$.
\item
$[Cr, \chi]_p^{FN} = 0 \in \Lambda^5 T_p^*M^7 \otimes T_pM^7$.
\item
$[\chi, \chi]_p^{FN} = 0 \in \Lambda^6 T_p^*M^7 \otimes T_pM^7$.
\end{enumerate}
\end{theorem}

\begin{proof} The equivalence of the first two statements is well known, see e.g. \cite{FG1982}. Proposition \ref{prop:Cr-chi} shows the equivalence of the first and the third, whereas Proposition \ref{prop:chi-chi} shows the equivalence of the first and the last statement. That $(\nabla \varphi)_p=0$ implies $[Cr, \chi]_p^{FN} = 0$ is immediate from the formula of the bracket in Proposition \ref{prop:FN}, and obviously, $[Cr, \chi]_p^{FN} = 0$ implies $\pi_7([Cr, \chi]_p^{FN}) = 0$.
\end{proof}

%%%%%%%%%%%%%%%%%%%%%%%%%%%%%%%%%%%%%%%%%%%%%%%%%%%%%%%%%
\section{Cross products and $\Sp$-structures}\label{sec:spin7}
%%%%%%%%%%%%%%%%%%%%%%%%%%%%%%%%%%%%%%%%%%%%%%%%%%%%%%%%%

%%%%%%%%%%%%%%%%%%%%%%%%%%%%%%%%%%%%%%%%%%%%%%%%%%%%%%%%%
\subsection{$\Sp$-structures and associated cross products}\label{subs:spin7str}
%%%%%%%%%%%%%%%%%%%%%%%%%%%%%%%%%%%%%%%%%%%%%%%%%%%%%%%%%

The exposition in this section mainly follows the references \cite{Bryant1987}, \cite{Fernandez1986}, \cite{HL1982}.

Let $M$ be an oriented $8$-manifold. A {\it $\Sp$-structure on $M$ }is a $4$-form $\Phi \in \Om^4(M)$ such that at each $p \in M$ there is a positively oriented basis $(e_\mu)_{\mu=0}^7$ of $T_pM$ with dual basis $(e^\mu)_{\mu=0}^7$ such that $\Phi_p \in \Lambda^4 T_pM$ is of the form 
\begin{eqnarray} \label{Phi4}
\Phi_p & := & e^{0123} + e^{0145} + e^{0167} + e^{0246} - e^{0257} - e^{0347} - e^{0356}\\
\nonumber & & + e^{4567} + e^{2367} + e^{2345} + e^{1357} - e^{1346} - e^{1256} - e^{1247}.
\end{eqnarray}

Throughout this section, we shall use Greek indices $\mu, \nu, \ldots$ to run over $0, \ldots, 7$, whereas Latin indices $i,j, \ldots$ range over $1, \ldots, 7$.

A basis $(e_\mu)$ of $T_pM$ whose dual basis $(e^\mu)$ satisfies (\ref{Phi4}) is called a {\em $\Sp$-frame}. Observe that if we define for a $\Sp$-frame $(e_\mu)$ the forms $\varphi_p$ and $\ast_7 \varphi_p$ on $V_p := \rmspan(e_i)_{i=1}^7 \subset T_pM$ as in (\ref{varphi}) and (\ref{varphi*}), then
\[
\Phi_p = e^0 \wedge \varphi_p + \ast_7 \varphi_p.
\]

The stabilizer of $\Phi_p$ is the group $\Sp$ acting on $T_pM$ via the spinor representation, and there is a unique $\Sp$-invariant Riemannian metric $g_\Phi$ on $M$ such that each $\Sp$-frame is orthonormal. In particular, $\Phi$ is self-dual w.r.t. $g_\Phi$. The set of all $\Sp$-frames forms a principal $\Sp$-bundle
\[
\Sp_M = \Sp_{(M, \Phi)} \longrightarrow M,
\]
and again, for each $\Sp$-module $W$ we obtain the associated vector bundle
\begin{equation} \label{eq:def-Spin7-bundles}
W(M) := \Sp_M \times_{\Sp} W \longrightarrow M.
\end{equation}
For instance, if we denote the $k$-dimensional irreducible $\Sp$-module by $W_k$ (in case the dimension uniquely specifies this module), then
\[
W_8(M) \cong TM \cong T^*M.
\]

It is well known that the action of $\Sp$ on $W_8$ is transitive on the unit sphere $S^7 \subset W_8$, and the stabilizer of an element is isomorphic to $G_2 \subset \Sp$. In analogy of the products $Cr$ and $\chi$ on manifolds with a $G_2$-structure in Definition \ref{def:Crchi}, we define on a $\Sp$-manifold $M$ a triple product as follows.

\begin{definition} \label{def:P}
Let $(M, \Phi)$ be manifold with a $\Sp$-structure, and let $g = g_\Phi$ be the induced Riemannian metric. Then the $TM$-valued form $P =P_\Phi \in \Om^3(M, TM)$ is defined by 
\[
P_\Phi := -\delta_{g_\Phi}(\Phi),
\]
and is called the {\em $3$-fold cross product on $M$}. That is, for $x,y,z,w \in TM$ we have
\begin{equation} \label{def-P}
g(P(x,y,z), w) = \Phi (x,y,z,w).
\end{equation}
\end{definition}

We shall usually suppress the indices $\Phi$ for $g$ and $P$ if it is clear from the context which $\Sp$-structure $\Phi$ is used.

%%%%%%%%%%%%%%%%%%%%%%%%%%%%%%%%%%%%%%%%%%%%%%%%%%%%%%%%%
\subsection{$\Sp$-representations}\label{subs:decomspin7}
%%%%%%%%%%%%%%%%%%%%%%%%%%%%%%%%%%%%%%%%%%%%%%%%%%%%%%%%%

In this section, we shall discuss the decomposition of symmetric and anti-symmetric powers of $W_8$ as $\Sp$-modules. For its exterior powers, we obtain the decompositions 
\begin{align}
\nonumber
\Lambda^0 W_8 &\cong \Lambda^8 W_8 \cong W_1, \qquad
& \Lambda^2 W_8 \cong \Lambda^6 W_8 &\cong W_7 \oplus W_{21},\\
\label{decom1-L-W8}
\Lambda^1 W_8 &\cong \Lambda^7 W_8 \cong W_8, \qquad
& \Lambda^3 W_8 \cong \Lambda^5 W_8 &\cong W_8 \oplus W_{48},\\
\nonumber
\Lambda^4 W_8 &\cong W_1 \oplus W_7 \oplus W_{27} \oplus W_{35}
\end{align}
where $\Lambda^k W_8 \cong \Lambda^{8-k} W_8$ via the Hodge-$\ast$. Again, we denote by $\Lambda^k_l W_8 \subset \Lambda^k W_8$ the subspace isomorphic to $W_l$ in the above notation.

Moreover, there are also irreducible decompositions of the symmetric powers of $W_7$ and $W_8$ as
\begin{equation} \label{decom-Sk-W7}
\odot^2 W_7 \cong W_1 \oplus W_{27}, \qquad \odot^2 W_8 \cong W_1 \oplus W_{35}
\end{equation}
into the induced $\Sp$-invariant metric and the trace free symmetric tensors; see \cite{Humphreys}.

\begin{lemma} \label{lem:lambdas}
Let $e^0 \in W_8$ be a unit vector, let $V_7 := e_0^\perp$ on which $\Sp$ acts as the double cover of $SO(7)$, so that $V_7 \cong W_7$ as a $\Sp$-module. Then the following maps are $\Sp$-equivariant embeddings.
\begin{equation} \label{def:lambda}
\lambda^k: W_7 \longrightarrow \Lambda^k W_8, \qquad 
\begin{array}{ll} \lambda^2(v) := e^0 \wedge v^\flat + (\imath_v \varphi)\\[2mm]
\lambda^4(v) := e^0 \wedge (\imath_v \ast_7 \varphi) - v^\flat \wedge \varphi\\[2mm]
\lambda^6(v) := \Phi \wedge \lambda^2(v) = 3 \ast \lambda^2(v)
\end{array}
\end{equation}
Here $\ast$ and $\ast_7$ denote the Hodge-$\ast$ in $W_8$ and $V_7$, respectively.
\end{lemma}

\begin{proof} The decompositions in (\ref{decom-Sk-W7}) imply that there are $\Sp$-equivariant maps $\lambda^k: W_7 \to \Lambda^k W_8$, and these are unique up to rescaling.

The equivariance of $\lambda^2$ follows from \cite[p.68]{SW2017}, and thus, $\Phi \wedge \lambda^2(v) \in \Lambda^6_7 W_8$, whence $\Phi \wedge \lambda^2(v) = 3 \ast \lambda^2(v)$ follows from \cite[Theorem 9.8]{SW2017}. This shows the statement on $\lambda^6$.

By \cite[Theorem 9.8]{SW2017}, $\Lambda^4_7 W_8$ is the infinitesimal orbit of $\Phi$ under the action of $\so(W_8) \cong \Lambda^2 W_8$. That is,
\[
\Lambda^4_7 W_8 = \{(u^\flat \wedge v^\flat) \cdot \Phi \mid u, v \in W_8\} = \{ u^\flat \wedge (\imath_v \Phi) - v^\flat \wedge (\imath_u \Phi) \mid u, v \in W_8\}.
\]
Setting $u := e_0$ and picking $v \in e_0^\perp \cong W_7$ for a $\Sp$-frame $(e_\mu)$, it follows that the image of $\lambda^4$ equals $\Lambda^4_7 W_8$, and since $\lambda^4$ is evidently $G_2$-equivariant, it must coincide with the $\Sp$-equivariant map $W_7 \to \Lambda^4_7 W_8$.
\end{proof}

From this lemma, we obtain the following descriptions of the decompositions, which essentially recapitulates \cite[Theorem 9.8]{SW2017}.

\begin{align}
\nonumber
\Lambda^k_7 W_8 &= \{\lambda^k(v) \mid v \in V_7\} \quad \mbox{for $k = 2,4,6$},\\
\nonumber
\Lambda^2_{21} W_8 &= \{ \alpha^2 \in \Lambda^2 W_8 \mid \alpha^2 \wedge \Phi \wedge \lambda^2(v) = 0\; \mbox{for all $v \in V_7$}\},\\
\nonumber
\Lambda^6_{21} W_8 &= \{ \alpha^6 \in \Lambda^6 W_8 \mid \alpha^6 \wedge \lambda^2(v) = 0\; \mbox{for all $v \in V_7$}\},\\
\nonumber
\Lambda^3_8 W_8 &= \{ \imath_a \Phi \mid a \in W_8\},\\
\nonumber
\Lambda^5_8 W_8 &= \{ a^\flat \wedge \Phi \mid a \in W_8\},\\
\label{decom-L-W8}
\Lambda^3_{48} W_8 &= \{ \alpha^3 \in \Lambda^3 W_8 \mid \Phi \wedge \alpha^3 = 0\},\\
\nonumber
\Lambda^5_{48} W_8 &= \{ \alpha^5 \in \Lambda^3 W_8 \mid \Phi \wedge \ast \alpha^5 = 0\},\\
\nonumber
\Lambda^4_1 W_8 &= \R \Phi,\\
\nonumber
\Lambda^4_{27} W_8 &= \rmspan\{\lambda^2(v) \wedge \lambda^2(w) \mid v,w \in V_7, \la v,w \ra = 0\}\\
\nonumber
\Lambda^4_{35} W_8 &= \{ \alpha^4 \mid \ast \alpha^4 = - \alpha^4\}.
\end{align}

We also recall the decomposition of the tensor product
\begin{equation} \label{eq:W8W7}
Lin(W_8, W_7) := W_8^* \otimes W_7 = W_8 \oplus W_{48}.
\end{equation}
Here, the summand isomorphic to $W_8$ is given as
\begin{equation} \label{eq:W8W7-summand}
\{ (\imath_a \lambda^2(e_i)) \otimes e_i \mid a \in W_8\},
\end{equation}
where the sum is taken over an orthonormal basis $(e_i)$ of $V_7 \cong W_7$. Finally, we define the $\Sp$-invariant tensor $\sigma \in (W_8 \otimes W_7 \otimes W_8 \otimes W_7)^*$ by
\begin{equation} \label{eq:sigma}
\sigma(a, u, b, v) := \frac12 \ast (a^\flat \wedge b^\flat \wedge \lambda^4(u) \wedge \lambda^2(v)).
\end{equation}

Contraction with the inner products on $W_7$ and $W_8$ induces a $\Sp$-equivariant map
\begin{align} \label{eq:phi-sigma}
\phi_\sigma: Lin(W_8, W_7) &\longrightarrow Lin(W_8, W_7)\\
\nonumber \phi_\sigma(A)(a) &:= \sigma(a, A(e_\mu), e_\mu, e_i)\; e_i.
\end{align}
We calculate
\begin{align}
\nonumber
\lambda^4(u) \wedge \lambda^2(v) & =\; (e^0 \wedge (\imath_u \ast_7 \varphi) - u^\flat \wedge\varphi) \wedge (e^0 \wedge v^\flat + \imath_v \varphi)\\
\nonumber
& =\; e^0 \wedge (\imath_u \ast_7 \varphi) \wedge (\imath_v \varphi) + e^0 \wedge u^\flat \wedge v^\flat \wedge \varphi - u^\flat \wedge \varphi \wedge (\imath_v \varphi)\\ 
\label{eq:l4l2} & \stackrel{(\ref{eq:form13}), (\ref{eq:form14})}=\; 2 u^\flat \wedge v^\flat \wedge (e^0 \wedge \varphi - \ast_7 \varphi) - 2 e^0 \wedge (\ast_7 ( u^\flat \wedge v^\flat )).
\end{align}

\begin{lemma} \label{lem:phisigma}
The map $\phi_\sigma$ has eigenvalues $-1$ and $6$ with multiplicity $48$ and $8$, respectively.
\end{lemma}

\begin{proof} Observe that the $\Sp$-invariant inner products on $W_7$ and $W_8$ induce an inner product on $Lin(W_8, W_7) = W_8^* \otimes W_7$ for which $(e^\mu \otimes e_i)$ is an orthonormal basis whenever $(e_\mu)$ is an orthonormal basis of $W_8$ so that $V_7 = e_0^\perp$ is spanned by $(e_i)$. This induced inner product satisfies
\[
\la \phi_\sigma(e^\mu \otimes e_i), e^\nu \otimes e_j \ra_{Lin(W_8, W_7)} = \la \phi_\sigma(e^\mu \otimes e_i)(e_\nu), e_j\ra_{W_7} = \sigma(e_\nu, e_i, e_\mu, e_j),
\]
and since $\sigma(e_\mu, e_i, e_\mu, e_j) = 0$, it follows that the matrix representation of $\phi_\sigma$ w.r.t. the basis $(e^\mu \otimes e_i)$ has $0$'s on the diagonal, whence $tr(\phi_\sigma) = 0$. Furthermore, $\phi_\sigma$ is self-adjoint since $\sigma(a,u,b,v) = \sigma(b,v,a,u)$ by (\ref{eq:sigma}) 
and (\ref{eq:l4l2}), whence has real eigenvalues.

Decomposing $Lin(W_8, W_7) \stackrel{(\ref{eq:W8W7})}\cong W_8 \oplus W_{48}$, (\ref{eq:W8W7-summand}) implies that the elements in the summand congruent to $W_8$ are given by the maps
\[
A_a: W_8 \longrightarrow W_7, \qquad A_a(b) := \la (\imath_a \lambda^2(e_i))^\#, b\ra e_i
\]
for a fixed $a \in W_8$. In order to calculate $\phi_\sigma(A_a)$, observe that $\Sp$ acts transitively on the unit sphere, whence we may assume w.l.o.g. that $a = e_0$, so that
\[
A_{e_0}(b) = \la e_i, b\ra e_i = pr_{e_0^\perp}(b),
\]
where $pr_{e_0^\perp}: W_8 \rightarrow e_0^\perp = V_7$ is the orthogonal projection.
Thus,
\begin{align*}
\phi_\sigma(A_{e_0})(b) &=\; \sigma(b, A_{e_0}(e_\mu), e_\mu, e_i)\; e_i = \sigma(b, e_j, e_j, e_i)\; e_i\\
\stackrel{(\ref{eq:l4l2})}=&\; 
\ast \Big( b^\flat \wedge e^j \wedge \Big(e^j \wedge e^i \wedge (e^0 \wedge \varphi - \ast_7 \varphi) - e^0 \wedge (\ast_7 e^{ji})\Big) \Big) e_i\\
=&\; \ast \Big( e^0 \wedge b^\flat \wedge e^j \wedge (\ast_7 e^{ij}) \Big) e_i\\
=&\; (1 - \delta_{ij}) \ast (e^0 \wedge b^\flat \wedge \ast_7 e^i) e_i\\
=&\; 6 \la b, e_i\ra e_i = 6 A_{e_0}(b),
\end{align*}
so that $\phi_\sigma(A_a) = 6 A_a$ for all $A_a$. By Schur's lemma and since $\phi_\sigma$ is self-adjoint, $\phi_\sigma|_{W_{48}} = c Id_{W_{48}}$ for some $c \in \R$, 
whence
\[
0 = tr(\phi_\sigma) = 6 \dim W_8 + c \dim W_{48},
\]
and from this, $c = -1$ and the lemma follows.
\end{proof}

For a manifold with a $\Sp$-structure $(M, \Phi)$ and induced metric $g = g_\Phi$, the covariant derivative $g_\Phi$ $\nabla_v \Phi$ w.r.t. the Levi-Civita connection is contained in the infinitesimal orbit of $\so(T_pM, g_p)$ \cite{Bryant1987} and hence in $\Lambda^4_7 T^\ast M$. That is, there is a section $T \in \Om^1(M, W_7(M)) = \Gamma(Lin(TM, W_7 (M)))$ such that
\begin{equation} \label{eq:nabla-Spin7}
\nabla_v \Phi = \lambda^4(T(v)) = e^0 \wedge (\imath_{T(v)} \ast \varphi) - (T(v))^\flat \wedge \varphi
\end{equation}
with the map $\lambda^4: W_7 \to \Lambda^4_7 T_pM$ from (\ref{def:lambda}).
In analogy to Definition \ref{def:torsionG2}, we use the following terminology.

\begin{definition} \label{def:torsionSpin7}
Let $(M, \Phi)$ be a manifold with a $\Sp$-structure. The section $T \in \Om^1(M, W_7(M))$ for which (\ref{eq:nabla-Spin7}) holds is called the {\em torsion endomorphism }of the $\Sp$-structure.
\end{definition}

%From (\ref{eq:nabla-Spin7}), it is now evident that $(M, \Phi)$ is torsion free at $p \in M$ (i.e., $T_p = 0$) iff $d\Phi_p = 0$ (cf. \cite{Fernandez1986}).

%%%%%%%%%%%%%%%%%%%%%%%%%%%%%%%%%%%%%%%%%%%%%%%%%%%%%%%%%
\subsection{The Fr\"olicher-Nijenhuis brackets on a manifold with a $\Sp$-structure}
%%%%%%%%%%%%%%%%%%%%%%%%%%%%%%%%%%%%%%%%%%%%%%%%%%%%%%%%%

Recall the section $P = - \delta_g \Phi \in \Om^3(M, TM)$ on a manifold with a $\Sp$-structure $(M, \Phi)$ from Definition \ref{def:P}. We wish to relate its Fr\"olicher-Nijenhuis bracket to its torsion. In order to do this, recall that $[P,P]^{FN} \in \Om^6(M, TM)$.

Due to the decomposition $\Lambda^6 W_8 = \Lambda^6_7 W_8 \oplus \Lambda^6_{21} W_8$ as a $G_2$-module, we may decompose
\[
\Om^6(M, TM) = \Gamma(M, \Lambda^6_7 T^*M \otimes TM) \oplus \Gamma(M, \Lambda^6_{21} T^*M \otimes TM),
\]
and we denote the projections onto the two summands by $\pi_7$ and $\pi_{21}$, respectively. 

\begin{proposition} \label{prop:PP}
Let $(M, \Phi)$ be a manifold with a $\Sp$-structure with the torsion endomorphism $T \in \Om^1(M, W_7(M))$ from (\ref{eq:nabla-Spin7}), and let $P = - \delta_g \Phi \in \Om^3(M, TM)$ be as before. Then for $p \in M$,
\begin{equation} \label{eq:pi7-Phi}
\pi_7([P, P]_p^{FN}) = -\dfrac23 \Phi \wedge \lambda^2\Big( (4 T_p + \phi_\sigma(T_p))(e_\mu) \Big) \otimes e_\mu.
\end{equation}
In particular, $\pi_7([P, P]_p^{FN}) = 0$ iff $T_p = 0$.
\end{proposition}

\begin{proof}
By Proposition \ref{prop:FN}, $[P,P]^{FN}_p = \gamma_\mu \otimes e_\mu$, where $\gamma_\mu \in \Lambda^6 T_p^*M$ is given by
\begin{equation}
\gamma_\mu = 2 ((\imath_{e_\nu} \Phi) \wedge \imath_{e_\mu} (\nabla_{e_\nu} \Phi) + (\imath_{e_\mu} \imath_{e_\nu} \Phi) \wedge e^\rho \wedge (\imath_{e_\nu} \nabla_{e_\rho}\Phi)).
\end{equation}

If we decompose $\gamma_\mu = \Phi \wedge \lambda^2(v_\mu) + \gamma_\mu^{21}$ with $\gamma_\mu^{21} \in \Lambda^6_{21} T_p^*M$, then for any $v \in V_7 = e_0^\perp$ we have $\gamma_\mu^{21} \wedge \lambda^2(v) = 0$ by (\ref{decom-L-W8}) and hence,

\begin{align*}
\gamma_\mu \wedge \lambda^2(v) =&\; \Phi \wedge \lambda^2(v_\mu) \wedge \lambda^2(v)\\
 =&\; (e^0 \wedge \varphi + \ast_7 \varphi) \wedge (e^0 \wedge v_\mu^\flat + \imath_{v_\mu} \varphi) \wedge (e^0 \wedge v^\flat + \imath_v \varphi)\\
=&\; e^0 \wedge \varphi \wedge (\imath_{v_\mu} \varphi) \wedge (\imath_v \varphi) + \ast_7 \varphi \wedge e^0 \wedge v_\mu^\flat \wedge (\imath_v \varphi)\\
&\quad + \ast_7 \varphi \wedge (\imath_{v_\mu} \varphi) \wedge e^0 \wedge v^\flat\\
\stackrel{(\ref{eq:form-1}), (\ref{eq:form-0})}= &\; 6 \la v_\mu, v \ra \vol + 3 \la v_\mu, v \ra \vol + 3 \la v_\mu, v \ra \vol = 12 \la v_\mu, v \ra \vol.
\end{align*}
Thus,
\begin{equation}\label{eq:PPl}
\pi_7([P, P]_p^{FN}) = \frac1{12} \ast (\gamma_\mu \wedge \lambda^2(e_i))\; \Phi \wedge \lambda^2(e_i) \otimes e_\mu.
\end{equation}

For arbitrary $v \in V_7 = e_0^\perp$ we compute

\begin{align*}
\nonumber
\gamma_\mu \wedge \lambda^2(v) =&\; 2 ((\imath_{e_\nu} \Phi) \wedge \imath_{e_\mu} (\nabla_{e_\nu} \Phi) \wedge \lambda^2(v)\\
\nonumber
&\qquad + (\imath_{e_\mu} \imath_{e_\nu} \Phi) \wedge e^\rho \wedge (\imath_{e_\nu} \nabla_{e_\rho}\Phi) \wedge \lambda^2(v))\\
\nonumber
=&\; 2 ((\imath_{e_\nu} \Phi) \wedge \imath_{e_\mu} \lambda^4(T_p(e_\nu)) \wedge \lambda^2(v)\\
\nonumber
&\qquad + (\imath_{e_\mu} \imath_{e_\nu} \Phi) \wedge e^\rho \wedge (\imath_{e_{\nu}} \lambda^4(T_p(e_\rho))) \wedge \lambda^2(v))\\
\nonumber
\nonumber
\stackrel{(\ref{spinformula1}), (\ref{spinformula2})}=& 2 ((-4 \delta_{\mu \nu} \la T(e_\nu), v\ra \vol + e^{\nu \mu} \wedge \lambda^4(T_p(e_\nu)) \wedge \lambda^2(v))\\
\nonumber
& + (-12 \delta_{\mu \rho} \la T_p(e_\rho), v \ra \vol + e^{\rho \mu} \wedge \lambda^4(T_p(e_\rho)) \wedge \lambda^2(v))\\
\nonumber
=&\; 2 (-16 \la T_p(e_\mu), v\ra \vol - 2 e^{\mu \nu} \wedge \lambda^4(T_p(e_\nu)) \wedge \lambda^2(v)) \\
\nonumber
=&\; 2 \Big(-16 \la T_p(e_\mu), v\ra - 4 \sigma(e_\mu, T_p(e_\nu), e_\nu, v)\Big) \vol
\\
\stackrel{(\ref{eq:phi-sigma})}=&\; -8 \la (4 T_p + \phi_\sigma(T_p))(e_\mu), v\ra \vol,
\end{align*}
and this together with (\ref{eq:PPl}) implies (\ref{eq:pi7-Phi}) and completes the proof.
\end{proof}

With this, we are now ready to prove the following which immediately implies Theorem \ref{thm:Brackets-Spin-intro} from the introduction.

\begin{theorem} \label{thm:Brackets-Spin}
Let $(M^8, \Phi)$ be a manifold with a $\Sp$-structure $(M^8, \Phi)$, let $\nabla$ be the Levi-Civita connection of $g = g_\Phi$, and let $T \in \Om^1(M^8, W_7(M^8))$ be its torsion endomorphism 
defined in Definition \ref{def:torsionSpin7}. Then for every $p \in M^8$ the following are equivalent.
\begin{enumerate}
\item
$T_p = 0 \in T_p^*M^8 \otimes W^7 (M)_p.$
\item
The $\Sp$-structure is torsion-free at $p$, i.e., $(\nabla \Phi)_p = 0$.
\item
$\pi_7([P, P]_p^{FN}) = 0 \in \Lambda^6_7 T_p^*M^8 \otimes T_pM^8$.
\item
$[P, P]_p^{FN} = 0 \in \Lambda^6 T_p^*M^8 \otimes T_pM^8$.
\end{enumerate}
\end{theorem}

\begin{proof}
The equivalence of the first two statements was shown in \cite{Fernandez1986}. Also, $T_p = 0$ implies $(\nabla \Phi)_p = 0$, whence by Proposition \ref{prop:FN}, $[P, P]_p^{FN} = 0$, and this trivially implies $\pi_7([P, P]_p^{FN}) = 0$.

By (\ref{eq:pi7-Phi}), $\pi_7([P, P]_p^{FN}) = 0$ iff $4 T_p + \phi_\sigma(T_p) = 0$, and since $\phi_\sigma$ does not have $-4$ as an eigenvalue by Lemma \ref{lem:phisigma}, this implies that $T_p = 0$.
\end{proof}

%%%%%%%%%%%%%%%%%%%%%%%%%%%%%%%%%%%%%%%%%%%%%%%%%%%%%%%%%
\section{The 16 classes of $G_2$- and 4 classes of $\Sp$-structures} \label{subs:16classesG2}
%%%%%%%%%%%%%%%%%%%%%%%%%%%%%%%%%%%%%%%%%%%%%%%%%%%%%%%%%

In this section, we shall interpret the classification of $G_2$-structures and of $\Sp$-structures (\cite{FG1982} and \cite{Fernandez1986}) in terms of the Fr\"olicher-Nijenhuis bracket. 

For the $G_2$-case, this classification is given by determining which components of the torsion endomorphism $T$ vanish, where $T$ is regarded as a section of the endomorphism bundle $\End(TM^7)$ which is $G_2$-equivariantly isomorphic to
\begin{equation} \label{eq:decomp-endo}
V_7(M^7) \otimes V_7(M^7) \cong V_1(M^7) \oplus V_7(M^7) \oplus V_{14}(M^7) \oplus V_{27}(M^7).
\end{equation}
Since this decomposition has $4$ summands, the classification consists of $2^4 = 16$ cases.

Observe that both $\Lambda^5_7 T^\ast M^7 \otimes TM^7$ and $\Lambda^6 T^\ast M^7 \otimes TM^7$ are $G_2$-equivariantly isomorphic to $V_7(M^7) \otimes V_7(M^7)$, 
where explicit isomorphisms are given by\\\[
\begin{array}{cccccccc}
K: & \Lambda^5_7 T^\ast M^7 \otimes TM^7 & \ni & (\ast \varphi \wedge v^\flat) \otimes w & \longmapsto & v^\flat \otimes w & \in & T^\ast M^7 \otimes TM^7\\[2mm]
L: & \Lambda^6 T_p^\ast M^7 \otimes TM^7 & \ni & (\ast v^\flat) \otimes w & \longmapsto & v^\flat \otimes w & \in & T^\ast M^7 \otimes TM^7.
\end{array}
\]

If $(M^7, \varphi)$ is a manifold with a $G_2$-structure and the cross products $Cr$ and $\chi$, then we define the sections
\[
K_{\pi_7([Cr, \chi]^{FN})},\; L_{[\chi, \chi]^{FN}} \in \Gamma(\End(TM)).
\]
Therefore, by Propositions \ref{prop:Cr-chi}, \ref{prop:chi-chi} there are $G_2$-equivariant vector bundle isomorphisms
\begin{align*}
\tau_1, \tau_2: \End(TM^7) & \longrightarrow \End(TM^7)
\end{align*}
such that for the torsion endoromphism $T \in \Gamma(\End(TM))$ we have
\begin{equation} \label{eq:tor-FNG2}
{}\tau_1(T) = K_{\pi_7([Cr, \chi]^{FN})} \qquad \mbox{and} \qquad \tau_2(T) = L_{[\chi, \chi]^{FN}},
\end{equation}
where by a slight abuse of notation we denote the map $\tau_i: \Gamma(\End(TM)) \to \Gamma(\End(TM))$ applying $\tau_i$ pointwise by the same symbol.

For an element $A = a_{ij}e^i \otimes e_j \in \End(V_7)$ let us denote its skew-symmetrization by
\[
\sigma_A := a_{ij} e^{ij} \in \Lambda^2 V_7^\ast.
\]
With this notation, it follows from (\ref{eq:pi-Crchi}) and (\ref{eq:pi-chichi}) that $\tau_1$ and $\tau_2$ take the form
\begin{align*}
\tau_1(T) &= 2 \Big(T - 2 T^\top - tr(T) id \Big),\\
\tau_2(T) &= -4 (T + T^\top) + 6 \ast (e^i \wedge \sigma_T \wedge \varphi) \otimes e_i,
\end{align*}
summing over some basis $(e_i)$ with dual basis $(e^i)$.

The $G_2$-equivariance of $\tau_1$ and $\tau_2$ and (\ref{eq:tor-FNG2}) now implies that the $V_k(M)$-component of $T$ vanishes if and only if the $V_k(M)$-component of $K_{\pi_7([Cr, \chi]^{FN})}$ vanishes if and only if the $V_k(M)$-component of $L_{[\chi, \chi]^{FN}}$ vanishes. Since the cases in the Fernandez-Gray classification are determined by the vanishing of the components of $T$, we obtain the interpretation of these cases given in Table \ref{table-G21}.

\begin{table}[htbp]
\caption{} \label{table-G21}
\begin{tabular}{|c||c|c|c|} 
\hline
Classes & $\begin{array}{c}\mbox{Relation on}\\ A \in \{ T, K_{\pi_7([Cr, \chi]^{FN})}, L_{[\chi, \chi]^{FN}}\} \end{array}$ \\ \hline \hline
$V_1(M) \oplus V_7(M) \oplus V_{14}(M) \oplus V_{27}(M)$ & no relation  \\ \hline
$V_7(M) \oplus V_{14}(M) \oplus V_{27}(M)$ & $tr(A) = 0$ \\ \hline
$V_1(M) \oplus V_{14}(M) \oplus V_{27}(M)$ & $\sigma_A \in \Om^2_{14}(M)$ \\ \hline
$V_1(M) \oplus V_7(M) \oplus V_{27}(M)$    &$\sigma_A \in \Om^2_7(M)$ \\ \hline
$V_1(M) \oplus V_7(M) \oplus V_{14}(M)$    & $A + A^\top - \frac{2}{7} tr (A) {\rm id}_{TM} = 0$ \\ \hline
$V_{14}(M) \oplus V_{27}(M)$                & $\sigma_A \in \Om^2_{14}(M), \quad tr(A) = 0$\\ \hline
$V_7(M) \oplus V_{27}(M)$                   &$\sigma_A \in \Om^2_7(M), \quad tr(A) = 0$\\ \hline
$V_7(M) \oplus V_{14}(M)$                   & $A + A^\top=0$  \\ \hline
$V_1(M) \oplus V_{27}(M)$                   & $A - A^\top=0$ \\ \hline
$V_1(M) \oplus V_{14}(M)$                   & $A + A^\top - \frac{2}{7} tr (A) {\rm id}_{TM} = 0, \quad \sigma_A \in \Om^2_{14}(M)$\\ \hline
$V_1(M) \oplus V_7(M)$                      & $A + A^\top - \frac{2}{7} tr (A) {\rm id}_{TM} = 0, \quad \sigma_A \in \Om^2_7(M)$\\ \hline
$V_{27}(M)$                                  & $A - A^\top=0, \quad tr(A) = 0$\\ \hline
$V_{14}(M)$                                  & $A+A^\top = 0, \quad \sigma_A \in \Om^2_{14}(M)$ \\ \hline
$V_7(M)$                                     & $A+A^\top = 0, \quad \sigma_A \in \Om^2_7(M)$\\ \hline
$V_1(M)$                                     & $A = \frac{1}{7} tr (A) {\rm id}_{TM}$ \\ \hline
$\{ 0 \}$                                  & $A = 0$ \\ \hline
\end{tabular}
\end{table}

The interpretation of manifolds $(M^8, \Phi)$ with a $\Sp$-structure is analogous. Again, the torsion $T$ and the projection $\pi_7([P, P]^{FN})$ are sections of the $\Sp$-equivariantly isomorphic bundles $T^\ast M^8 \otimes W_7(M^8)$ and $\Lambda^6_7 T^\ast M^8 \otimes TM^8$, respectively, with an explicit identification given by
\[
H: \Lambda^6_7 T^\ast M^8 \otimes W_7(M^8) \ni \Phi \wedge (\lambda^2(a)) \otimes v \longmapsto 
v^\flat \otimes a \in T^\ast M^8 \otimes W_7(M^8),
\]
and if $(M^8, \Phi)$ is a manifold with a $\Sp$-structure and the $3$-fold product $P$, then by (\ref{eq:pi7-Phi})
\[
H_{\pi_7([P, P]^{FN})} = \tau_3(T), \qquad \mbox{where} \qquad \tau_3(T) = -\dfrac23 (4T + \phi_\sigma(T)).
\]
Here, by abuse of notation we regard $\phi_\sigma$ as the pointwise application of the map from (\ref{eq:phi-sigma}) to sections of $W_7(M^8) \otimes T^\ast M^8 \cong Lin(W_8(M^8), W_7(M^8))$.

By (\ref{eq:W8W7}), $W_7(M^8) \otimes T^\ast M^8$ can be decomposed as $W_8(M^8) \oplus W_{48}(M^8)$ whence by the $\Sp$-equivariance of $\tau_3$, the $W_k(M^8)$-component of $T$ vanishes if and only if the $W_k(M^8)$-component of $H_{\pi_7([P,P]^{FN})}$ does. Since the classification of Fern\'{a}ndez \cite{Fernandez1986} into $2^2 = 4$ different cases is given by the vanishing of the components of the torsion $T$, it follows that these cases can be also interpreted by the vanishing of the components of $H_{\pi_7([P,P]^{FN})}$, which leads to the interpretation of the classes of $\Sp$-manifolds given in Table \ref{table-Spin7}, where $pr_k: W_7(M^8) \otimes T^\ast M^8 \to W_k(M^8)$ is the canonical projection.

\begin{table}[htbp]
\caption{} \label{table-Spin7}
\begin{tabular}{|c||c|c|c|} 
\hline
Classes &$\begin{array}{c}\mbox{Relation on}\\ A \in \{ T, H_{\pi_7([P, P]^{FN})}\} \end{array}$ \\ \hline \hline
$W_8 \oplus W_{48}$ & no relation  \\ \hline
$W_{48}$                & $pr_8 (A) = 0$ \\ \hline
$W_8$                   & $pr_{48} (A) = 0$ \\ \hline
$\{ 0 \}$               & $A = 0$ \\ \hline
\end{tabular}
\end{table}

%\newpage

%%%%%%%%%%%%%%%%%%%%%%%%%%%%%%%%%%%%%%%%%%%%%%%%%%%%%%%%%
\section{Appendix}
%%%%%%%%%%%%%%%%%%%%%%%%%%%%%%%%%%%%%%%%%%%%%%%%%%%%%%%%%

In this appendix, we shall collect some of the formulas which we needed in the calculations in this paper. Most of them are known and can be found in a similar form e.g. in \cite[Lemma 4.37]{SW2017}, but we shall collect them here for the reader's convenience.

\begin{lemma} \label{formulas} For all $u,v,w,r \in V_7$ and any orthonormal basis $(e_i)$ of $V_7$ the following identities hold.
\begin{align}
\label{eq:form-0}
u^\flat \wedge (\imath_v \varphi) \wedge \ast \varphi =\;& 3 \la u, v\ra \vol\\
\label{eq:form-11}
u^\flat \wedge (\imath_v \ast \varphi) \wedge \varphi =\;& 4 \la u, v \ra \vol\\
\label{eq:form-1}
(\imath_u \varphi) \wedge (\imath_{v} \varphi) \wedge \varphi =\;& 6 \la u, v\ra \vol\\
\label{eq:form-2}
u^\flat \wedge (\imath_{v} \varphi) \wedge (\imath_{w} \varphi) \wedge (\imath_r \varphi) =\;& 2 \Big(\la v,w \ra \la u,r \ra\\
\nonumber
&+ \la u,v\ra \la w, r\ra + \la u,w\ra \la v, r \ra\Big) \vol\\
\label{eq:form-7}
(\imath_{u}\imath_{v} \ast \varphi) \wedge w^\flat \wedge \ast \varphi =\;& -2 u^\flat \wedge v^\flat \wedge w^\flat \wedge \ast \varphi\\
\label{eq:form-3}
(\imath_{u}\imath_{v} \varphi) \wedge w^\flat \wedge (\imath_r \varphi) \wedge \varphi =\;& 2 (\la v,w\ra \la u,r\ra - \la u,w\ra \la v,r\ra) \vol \\ \nonumber & - 2 u^\flat \wedge v^\flat \wedge w^\flat \wedge r^\flat \wedge \varphi\\
\label{eq:form-4}
(\imath_{u} \imath_{v} \ast \varphi) \wedge (\imath_{w} \ast \varphi) \wedge (\imath_r \varphi) =\;& 2(\la v,w\ra \la u,r\ra - \la u,w\ra \la v,r\ra) \vol \\ \nonumber & + u^\flat \wedge v^\flat \wedge w^\flat \wedge r^\flat \wedge \varphi\\
\label{eq:form-5}
u^\flat \wedge v^\flat \wedge (\imath_w \ast \varphi) \wedge (\imath_r \varphi) =\;&2(\la v,w\ra \la u,r\ra - \la u,w\ra \la v,r\ra) \vol \\ \nonumber & + u^\flat \wedge v^\flat \wedge w^\flat \wedge r^\flat \wedge \varphi\\
\label{eq:form-6}
u^\flat \wedge v^\flat \wedge (\imath_w \imath_r \ast \varphi) \wedge \varphi =\;&2(\la v,w\ra \la u,r\ra - \la u,w\ra \la v,r\ra) \vol \\ \nonumber & - u^\flat \wedge v^\flat \wedge w^\flat \wedge r^\flat \wedge \varphi\\
%\label{eq:form12}
%(\imath_u \ast \varphi) \wedge (\imath_v \ast \varphi) =\;& 2 u^\flat \wedge v^\flat \wedge \ast \varphi\\
\label{eq:form13}
(\imath_u \ast \varphi) \wedge (\imath_v \varphi) =\;& 
-2 \ast ( u^\flat \wedge v^\flat )+ u^\flat \wedge v^\flat \wedge \varphi\\
\label{eq:form14}
\varphi \wedge (\imath_u \varphi) =\;& 2 u^\flat \wedge \ast \varphi\\
%\label{eq:form15}
%\ast \varphi \wedge (\imath_u \varphi) =\;& 3 \ast u^\flat\\
%\label{eq:form16}
%\varphi \wedge (\imath_u \ast \varphi) =\;& -4 \ast u^\flat\\
\label{formulas2} (\imath_{e_i} \varphi)^2 =\;& 6 \ast \varphi\\
\label{formulas-3} (\imath_u \imath_{e_i} \ast \varphi) \wedge (\imath_v \imath_{e_i} \ast \varphi) =\;& 2 (\imath_u \varphi) \wedge (\imath_v \varphi)\\
\label{formulas-4}
(\imath_u \imath_{e_i} \ast \varphi) \wedge (\imath_{e_i} \varphi) =\;& 3 u^\flat \wedge \varphi%\\
%\label{formulas-5}
%(\imath_u \imath_{e_i} \ast \varphi) \wedge (\imath_{e_i} \ast \varphi) =\;& 4 u^\flat \wedge \ast \varphi = 4 \ast (\imath_u \varphi).
\end{align}
\end{lemma}

\begin{proof}
For the proof of these identities, observe that the left hand side of each equation is a $G_2$-invariant element of some tensor power of $V_7$, and therefore it has to be a linear combination of the summands on the right hand side; the coefficients of this linear combination then can be determined by using the explicit formulas for $\varphi$ and $\ast \varphi$ in (\ref{varphi}) and (\ref{varphi*}).

To pick one explicit example which is not among the identities shown in \cite{SW2017}, observe that the left hand side of (\ref{eq:form-2}) is an element of $(V_7 \otimes \odot^3 V_7)^{G_2}$. Since $\odot^3 V_7 \cong V_7 \oplus V_{77}$, we have $\dim (V_7 \otimes \odot^3 V_7)^{G_2} = 1$, and there is one $G_2$-invariant element of $V_7 \otimes \odot^3 V_7$ given by deriving the square of the scalar product which lies in $\odot^4 V_7$. Thus,
\begin{align}
u^\flat \wedge (\imath_{v} \varphi) \wedge (\imath_{w} \varphi) \wedge (\imath_r \varphi) =\;& c \Big(\la v,w \ra \la u,r \ra\\
\nonumber &+ \la u,v\ra \la w, r\ra + \la u,w\ra \la v, r \ra\Big) \vol.
\end{align}
for some constant $c \in \R$. Now setting $u = v = w = r =: e_1$ and using (\ref{varphi}) implies that $c = 2$, showing (\ref{eq:form-2}).

The remaining identities are shown in a similar fashion.
\end{proof}

%We also get the following decomposition of $\Lambda^2 V_7$.
The following two decompositions of $G_2$- and $\Sp$-representations is also well known, cf. \cite[Theorem 8.5, 9.8]{SW2017}, \cite[(4.7), (4.8)]{Kar2005}.

\begin{lemma} \label{Lambda2-V7}
Decompose $\Lambda^2 V_7^\ast = \Lambda^2_7 V_7^\ast \oplus \Lambda^2_{14} V_7^\ast$ according to (\ref{decom-L-V7}). Then
\begin{align*}
\Lambda^2_7 V_7^\ast &= \{ \alpha^2 \in \Lambda^2 V_7^\ast \mid \ast(\alpha^2 \wedge \varphi) = 2 \alpha^2\}, \quad \mbox{and}\\
\Lambda^2_{14} V_7^\ast &= \{ \alpha^2 \in \Lambda^2 V_7^\ast \mid \ast(\alpha^2 \wedge \varphi) = - \alpha^2\}.
\end{align*}
In particular,
\begin{equation} \label{eq:describe-Lamb2V7}
\begin{array}{lll}
\Lambda^2_7 V_7^\ast &= &\{ \alpha^2 + \ast(\alpha^2 \wedge \varphi) \mid \alpha^2 \in \Lambda^2 V_7^\ast\}, \quad \mbox{and}\\[2mm]
\Lambda^2_{14} V_7^\ast &= &\{ 2 \alpha^2 - \ast(\alpha^2 \wedge \varphi) \mid \alpha^2 \in \Lambda^2 V_7^\ast\}.
\end{array}
\end{equation}
\end{lemma}

%\begin{proof} The map $\tau: \alpha^2 \mapsto \ast(\alpha^2 \wedge \varphi)$ is a $G_2$-equivariant endomorphism of $\Lambda^2 V_7^\ast$. As $\Lambda^2_7 V_7^\ast$, $\Lambda^2_{14} V_7^\ast$ 
%do not admit $G_2$-invariant complex structures, it follows from Schur's Lemma that there are constants $\lambda_7, \lambda_{14} \in \R$ such that
%\[
%\tau|_{\Lambda^2_k V_7^\ast} = \lambda_k Id_{\Lambda^2_k V_7^\ast}.
%\]
%Moreover, observe that for the inner product we have
%\[
%(\tau(\alpha^2), \beta^2) \vol = (\ast(\alpha^2 \wedge \varphi), \beta^2) \vol = \alpha^2 \wedge \varphi \wedge \beta^2,
%\]
%which implies that the matrix representation of $\tau$ w.r.t. to the orthonormal basis $e^{ij}$ of $\Lambda^2 V_7^\ast$ has only $0$'s on the diagonal, whence $tr(\tau) = 0$. Thus,
%\[
%0 = tr(\tau) = \lambda_7 \dim \Lambda^2_7 V_7^\ast + \lambda_{14} \dim \Lambda^2_{14} V_7^\ast = 7 \lambda_7 + 14 \lambda_{14},
%\]
%so that $\lambda_7 = - 2 \lambda_{14}$. In order to determine $\lambda_7$, we let $\alpha^2 \in \Lambda^2_7 V_7^\ast$, so that by (\ref{decom-L-V7}) there is a $v \in V_7$ such that $\alpha^2 = \imath_v \varphi$ and hence,
%\[
%\ast(\alpha^2 \wedge \varphi) = \ast((\imath_v \varphi) \wedge \varphi) \stackrel{(\ref{eq:form14})}= 2 \ast(v^\flat \wedge \ast \varphi) = 2 \ast^2 (\imath_v \varphi) = 2 \alpha^2.
%\]
%Thus, $\lambda_7 = 2$ and $\lambda_{14} = -1$, which shows the lemma.
%\end{proof}

%The following describes identities of representation of $\Sp$.

\begin{lemma} \label{Kar47-48}
Decompose $\Lambda^2 W_8^\ast = \Lambda^2_7 W_8^\ast \oplus \Lambda^2_{21} W_8^\ast$ according to (\ref{decom1-L-W8}). Then
\begin{align*}
\Lambda^2_7 W_8^\ast &= \{ \alpha^2 \in \Lambda^2 W_8^\ast \mid \Phi \wedge \alpha^2 = 3 \ast \alpha^2\}, \quad \mbox{and}\\
\Lambda^2_{21} W_8^\ast &= \{ \alpha^2 \in \Lambda^2 W_8^\ast \mid \Phi \wedge \alpha^2 = - \ast \alpha^2\}.
\end{align*}
\end{lemma}

%\begin{proof}
%Since the map $\alpha^2 \mapsto \ast (\Phi \wedge \alpha^2)$ is $\Sp$-equivariant, Schur's lemma implies that $\ast(\Phi \wedge \alpha^2) = c_i \alpha^2$ or, equivalently, $\Phi \wedge \alpha^2 = c_i \ast \alpha^2$ for $c_i \in \R$ and $\alpha^2 \in \Lambda^2_i W_8^\ast$. In Lemma \ref{lem:lambdas} it was already shown that $c_7 = 3$.
%
%Moreover, $\la \ast(\Phi \wedge \alpha^2), \alpha^2\ra = \ast(\Phi \wedge \alpha^2 \wedge \alpha^2\ra$. Thus, when writing this map w.r.t. the orthonormal basis $(e^{\mu\nu})$ of $\Lambda^2 W_8^\ast$, we have only $0$'s on the diagonal and hence, this map is trace free.
%
%But the trace of this map is given as $7 c_7 + 21 c_{21}$, and as $c_7 = 3$, it follows that $c_{21} = -1$.
%\end{proof}

We shall also need the following result.

\begin{lemma} For all $u, v \in W_7$ and $a, b \in W_8$ the following formulas hold:
\begin{align}
\nonumber
&(\imath_a \Phi) \wedge (\imath_{b} \lambda^4(u)) \wedge \lambda^2(v) =\\ 
\label{spinformula1}&\qquad -4 \la a, b \ra_{W_8} \la u,v \ra_{W_7} \vol + a^\flat \wedge b^\flat \wedge \lambda^4(u) \wedge \lambda^2(v),\\
\nonumber
&a^\flat \wedge (\imath_b \imath_{e_\mu} \Phi) \wedge (\imath_{e_\mu} \lambda^4(u)) \wedge \lambda^2(v) =\\
\label{spinformula2} &\qquad -12 \la a, b \ra_{W_8} \la u,v \ra_{W_7} \vol + a^\flat \wedge b^\flat \wedge \lambda^4(u) \wedge \lambda^2(v)
\end{align}
where in (\ref{spinformula2}) the sum is taken over an orthonormal basis $(e_\mu)$ of $W_8$.
\end{lemma}

\begin{proof} By (\ref{decom1-L-W8}) and (\ref{decom-Sk-W7}), the decomposition of $W_8 \otimes W_8$ and $W_7 \otimes W_7$ into $\Sp$-irreducible summands yields
\begin{align*}
W_8 \otimes W_8 &= \odot^2 W_8 \oplus \Lambda^2 W_8 \cong (W_1 \oplus W_{35}) \oplus (W_7 \oplus W_{21}),\\
W_7 \otimes W_7 &= \odot^2 W_7 \oplus \Lambda^2 W_7 \cong (W_1 \oplus W_{27}) \oplus W_{21},
\end{align*}
so that there are two inequivalent summands in common and hence, the space of $\Sp$-invariant tensors in $W_8 \otimes W_8 \otimes W_7 \otimes W_7$ is $2$-dimensional. Since the left hand sides of (\ref{spinformula1}) and (\ref{spinformula2}) describe such tensors, it follows that there must be constants $c_1, \ldots, c_4 \in \R$ such that
\begin{align}
\nonumber
&(\imath_a \Phi) \wedge (\imath_{b} \lambda^4(u)) \wedge \lambda^2(v) =\\ 
\label{spinformula1c}&\qquad c_1 \la a, b \ra_{W_8} \la u,v \ra_{W_7} \vol + c_2 a^\flat \wedge b^\flat \wedge \lambda^4(u) \wedge \lambda^2(v),\\
\nonumber
&a^\flat \wedge (\imath_b \imath_{e_\mu} \Phi) \wedge (\imath_{e_\mu} \lambda^4(u)) \wedge \lambda^2(v) =\\
\label{spinformula2c} &\qquad c_3 \la a, b \ra_{W_8} \la u,v \ra_{W_7} \vol + c_4 a^\flat \wedge b^\flat \wedge \lambda^4(u) \wedge \lambda^2(v).
\end{align}

In order to determine these constants, we first let $a = b := e_0$, so that
\begin{align*}
(\imath_{e_0} \Phi) \wedge (\imath_{e_0} \lambda^4(u)) \wedge \lambda^2(v) =& \varphi \wedge (\imath_u \ast_7 \varphi) \wedge (e^0 \wedge v^\flat + (\imath_v \varphi))\\
=& -e^0 \wedge v^\flat \wedge (\imath_u \ast_7 \varphi) \wedge \varphi\\
\stackrel{(\ref{eq:form-11})}=& -4 \la u,v \ra \vol,
\end{align*}
and from this, $c_1 = -4$ follows. Whence if we let $a := e_0$ and $b, u, v \in W_7 = e_0^\perp$ then
\begin{align*}
(\imath_{e_0} \Phi) \wedge (\imath_b \lambda^4(u)) \wedge \lambda^2(v) \stackrel{(\ref{def:lambda})}=& \varphi \wedge (-e^0 \wedge (\imath_b \imath_u \ast_7 \varphi) - \imath_b (u^\flat \wedge \varphi))\\ & \qquad \wedge (e^0 \wedge v^\flat + (\imath_v \varphi))\\
=& - \varphi \wedge e^0 \wedge (\imath_b \imath_u \ast_7 \varphi) \wedge (\imath_v \varphi)\\
& \qquad + \varphi \wedge u^\flat \wedge (\imath_b \varphi) \wedge e^0 \wedge v^\flat\\
\stackrel{(\ref{eq:form14})}=& 2 e^0 \wedge (\imath_b \imath_u \ast_7 \varphi) \wedge v^\flat \wedge \ast_7 \varphi\\
& \qquad + 2 b^\flat \wedge \ast_7 \varphi \wedge u^\flat \wedge e^0 \wedge v^\flat\\
\stackrel{(\ref{eq:form-7})}=& -2 e^0 \wedge b^\flat \wedge u^\flat \wedge v^\flat \wedge \ast_7 \varphi\\
\stackrel{(\ref{eq:l4l2})}=& e^0 \wedge b^\flat \wedge \lambda^4(u) \wedge \lambda^2(v),
\end{align*}
so that $c_2 = 1$ follows. Now substituting $a = b := e_0$ and $u = v := e_1$ into (\ref{spinformula2c}) and using the index $i$ to run from $1$ to $7$ yields
\begin{align*}
e^0 \wedge (-\imath_{e_i} \varphi) \wedge (\imath_{e_i} (-e^1 \wedge \varphi)) \wedge (\imath_{e_1} \varphi) =& e^0 \wedge (\imath_{e_i} \varphi) \wedge (\delta_{1i} \varphi - e^1 \wedge (\imath_{e_i} \varphi)) \wedge (\imath_{e_1} \varphi)\\
=& e^0 \wedge (\imath_{e_1} \varphi) \wedge (\imath_{e_1} \varphi) \wedge \varphi\\
&\qquad - e^0 \wedge e^1 \wedge (\imath_{e_i} \varphi) \wedge (\imath_{e_i} \varphi) \wedge \imath_{e_1} \varphi\\
\stackrel{(\ref{eq:form-1}),(\ref{formulas2})}=& 6 \vol - 6 e^0 \wedge e^1 \wedge \imath_{e_1} \varphi\wedge \ast \varphi\\
&\stackrel{(\ref{eq:form-0})}= -12 \vol,
\end{align*}
so that $c_3 = -12$ follows. Finally, for $a := e_0$, $u := e_1$ and $b, v \in W_7 = e_0^\perp$ (\ref{spinformula2c}) reads
\begin{align*}
e^0 \wedge (\imath_b \imath_{e_\mu} \Phi) \wedge &(\imath_{e_\mu} \lambda^4(e_1)) \wedge \lambda^2(v)\\
 =&\; e^0 \wedge (\imath_b \varphi) \wedge (\imath_{e_1} \ast_7 \varphi) \wedge (\imath_v \varphi)\\
& + e^0 \wedge (\imath_b \imath_{e_i} \ast_7 \varphi) \wedge (-\imath_{e_i} (e^1 \wedge \varphi)) \wedge (\imath_v \varphi)\\
\stackrel{(*)}=&\; 0 + e^0 \wedge (\imath_b \imath_{e_i} \ast_7 \varphi) \wedge (-\delta_{1i} \varphi + e^1 \wedge (\imath_{e_i} \varphi)) \wedge (\imath_v \varphi)\\
=&\; -e^0 \wedge (\imath_b \imath_{e_1} \ast_7 \varphi) \wedge \varphi \wedge (\imath_v \varphi)\\
& + e^{01} \wedge (\imath_b \imath_{e_i} \ast_7 \varphi) \wedge (\imath_{e_i} \varphi) \wedge (\imath_v \varphi)\\
\stackrel{(\ref{eq:form14}), (\ref{formulas-4})}=&\; -2 e^0 \wedge (\imath_b \imath_{e_1} \ast_7 \varphi) \wedge v^\flat \wedge \ast_7 \varphi + 3 e^{01} \wedge b^\flat \wedge \varphi \wedge (\imath_v \varphi)\\
\stackrel{(\ref{eq:form14}), (\ref{eq:form-7})}=&\; 4 e^0 \wedge b^\flat \wedge e^1 \wedge v^\flat \wedge \ast_7 \varphi + 6 e^{01} \wedge b^\flat \wedge v^\flat \wedge \ast_7 \varphi\\
=&\; -2 e^0 \wedge b^\flat \wedge e^1 \wedge v^\flat \wedge \ast_7 \varphi\\
\stackrel{(\ref{eq:l4l2})}=& e^0 \wedge b^\flat \wedge \lambda^4(e_1) \wedge \lambda^2(v),
\end{align*}
so that $c_4 = 1$ follows. At $(*)$ we have used that the map
\[
(u,v,w) \longmapsto \ast \Big((\imath_u \varphi) \wedge (\imath_v \ast_7 \varphi) \wedge (\imath_w \varphi)\Big)
\]
is a $G_2$-invariant element of $W_7 \otimes \odot^2 W_7$, and since by (\ref{decom-Sk-W7}) $\odot^2 W_7 \cong W_1 \oplus W_{27}$ has no irreducible component isomorphic to $W_7$, there is no such element other than $0$.
\end{proof}

\subsection*{Acknowledgement} A part of this project has been discussed during HVL's visit to the Osaka City University in December 2015. She thanks Professor Ohnita for his invitation to Osaka and
his hospitality. HVL and LS also thank the Max Planck Institute for Mathematics in the Sciences in Leipzig for its hospitality during extended visits. We also thank the referee for many helpful comments which enabled us to significantly improve the manuscript.

%%%%%%%%%%%%%%%%%%%%%%%%%%%%%%%%%%%%%%%%%%%%%%%%%%%%%%%%%

\end{document}